\documentclass[12pt,a4paper,reqno]{amsart}

\usepackage[centertags]{amsmath}

\usepackage[centertags]{amsmath}
\usepackage[margin=1in]{geometry}
\usepackage{amsmath}
\usepackage{amsfonts}
\usepackage{amssymb}
\usepackage{verbatim}
\usepackage{enumerate}

\usepackage{color}

\usepackage{graphicx}

\newcommand{\mP}{{\mathbb P}}

\newcommand{\calT}{{\mathcal T}}

\newcommand{\Legge}{\overset{  \mathcal{L}  }{=}}

\newcommand{\N}{\mathbb N}

\newcommand{\R}{\mathbb R}

\newcommand{\Z}{\mathbb Z}

\newtheorem{theorem}{Theorem}
\newtheorem{lemma}{Lemma}
\newtheorem{rem}{Remark}

\newtheorem{prop}{Proposition}

\newtheorem{definition}{Definition}
\newtheorem{corollary}{Corollary}

\title[Infinite paths with bounded sums]{Infinite paths on  a random environment of $\Z^2$ with
bounded and recurrent sums. }

\author{Emilio De Santis}
\author{Mauro Piccioni}
\address{University of Rome Sapienza, Department of Mathematics.
Piazzale Aldo Moro, 5, I-00185, Rome, Italy}
\email{desantis@mat.uniroma1.it}
\email{piccioni@mat.uniroma1.it}

\begin{document}

\begin{abstract}
This paper considers a random structure on the lattice $\Z^2$ of the following kind. To each edge $e$ a random variable $X_e$ is assigned, together with a random sign $Y_e \in \{-1,+1\}$. For an infinite self-avoiding path on $\Z^2$ starting at the origin consider the sequence of partial sums along the path. These are computed by summing the $X_e$'s for the edges $e$ crossed by the path, with a sign depending on the direction of the crossing. If the edge is crossed rightward or upward the sign is given by $Y_e$, otherwise by $-Y_e$. We assume that the sequence of $X_e$'s is i.i.d., drawn from an arbitrary common law and that the sequence of signs $Y_e$ is independent, with independent components drawn from a law which is allowed to change from horizontal to vertical edges. First we show that, with positive probability,  there exists an infinite self-avoiding path starting from the origin with bounded partial sums. Moreover the process of partial sums either returns to zero or at least it returns to any neighborhood of zero infinitely often. These results are somewhat surprising at the light of the fact that, under rather mild conditions, there exists with probability $1$ two sites with all the paths joining them having the partial sums exceeding  in absolute value any prescribed constant.
\end{abstract}

\maketitle

\section{Introduction}

The problems considered in the present paper have been inspired by those addressed in \cite{BM01}. In \cite{BM01} the lattice $\Z^2$ is endowed with an \emph {environment} made by an i.i.d. field $(X_i, i \in \Z^2)$ of \emph{sign} variables, i.e. variables assuming values $\pm 1$, placed on the vertices of the lattice. The authors consider the sequence of partial sums of these variables made along infinite self-avoiding paths. They prove that 
when the parameter $p=\mathbb {P} (X_0=+1)$ is close to $1/2$ paths with  partial sums bounded by a  positive constant $C$ exist with probability $1$, and moreover the process of partial sums  returns to zero infinitely often (indeed every $42$ steps); at the contrary, when $p$ is close to $0$ or $1$ no path with partial sums bounded  exists, almost surely. Problems of this kind are clearly related to the area of first passage percolation, see \cite{50anni} for a review and \cite{BvdHG} \cite{Damron17} for recent results.

In our model the environment consists of a field of independent random variables placed on the edges of $\Z^2$, that are of two kind: a real random variable $X_e$, drawn from an arbitrary law $\mathcal{L} =\mathcal{L} (X_e)$, and a sign random variable $Y_e$, with $\mathbb{P} (Y_e=+1 )$ equal either to $p_o$ or to $p_v$, in case the edge $e$ is horizontal or vertical, respectively. When a path crosses the edge $e$ upward or rightward, then $Y_eX_e$ is added to the current partial sum of the path, whereas when it crosses $e$ downward or leftward, then $-Y_eX_e$ is added to the sum. 
It may be suggestive to consider the random variables $ X_e $ to be positive as it happens in problems of first passage percolation, where however it is not possible to change the contribution of an edge by crossing it in the opposite direction. This is a possible explanation of the fact that our main result (Theorem 1) establishes the a.s. existence of a path with bounded sums irrespectively of $\mathcal{L}$ and the parameters $p_o$ and $p_v$, so there is no phase transition behaviour as in \cite{BM01}. Concerning the recurrence of the process of partial sums to zero, for general $\mathcal{L}$ 
it is not possible to get more  that zero is an accumulation point (Theorem 2).

A suggestive interpretation of this model is the following: the random variables $Y_eX_e$ represent a field of \emph{slopes} on the edges of $\Z^2$. One can interpret a partial sum on a path as an \emph{height}, which is updated by adding the slope of each visited edge, that has to be reversed when the edge is crossed in opposition with the standard orientation of the two axes. With this interpretation one can see some similarities with the model introduced in \cite{Haggstrom}.  Loosely speaking, in this paper an environment of i.i.d. random variables is considered on the vertices of a quasi-transitive graph. A certain random payoff is obtained as a function of the environment, depending on the choice of an edge sequence. The authors classify the support of the optimal payoff in terms of the structure of the underlying graph. As described below, our Theorem 1 and Theorem 3 can also be seen as results about the support of the optimal value of some payoff which is a function of the environment and depends on a selection of paths.

Our main result about the model is that, irrespectively of $\mathcal{L}$, $p_o$ and $p_v$, with probability $1$ there exists a self-avoiding path $\gamma^*$ with partial sums bounded by a suitable  positive constant $C$. This path is realized through a suitable construction of blocks of edges induced by a tessellation of the underlying Euclidean plane, on which a structure of oriented graph is specified. More precisely the path $\gamma^*$ is made by concatenating paths living in an oriented path of \emph{good} blocks. The a.s. existence of such an oriented path is established using a classical result  which is applied to $1$-dependent fields (see \cite{Grimmett,LSS}). Indeed, taking blocks suitably large, we can made the probability that a block is good arbitrarily close to $1$.

The construction of $\gamma^*$ allows also to address the question of the recurrence to zero of its partial sums. Indeed, the partial sums of $\gamma^*$ at the exit of each block are shown to be an homogeneous Markov process. For $p_o$ and $p_v$ \emph{non-degenerate}, i.e. lying inside the open interval $(0,1)$, we can establish that $0$ is either recurrent or at least topologically recurrent.

 This situation is somewhat surprising at the light of another result established in the paper, concerning the a.s. existence of a pair of sites $u$ and $v$ in $\Z^2$ with the property that all the paths joining them have partial sums exceeding any given positive constant $C$. In the non-degenerate case the set of $\mathcal{L}$'s for which this property holds is characterized to be the set of laws with a non-zero atom or an unbounded support.

 Finally, we briefly outline the structure of the paper. In Section 2 the basic definitions are introduced and the results are stated. In Section 3 two lemmas are presented, which are fundamental for the subsequent proofs. They have been singled out since they may have an independent interest. In Section 4 the basic block construction used in the proofs is introduced. Finally, in Section 5 the proof of the results stated in Section 2 are provided.

\section{Main results}

Before stating our main results we recall the precise definition of the mathematical objects we are interested in.
The square lattice $\mathbb{L}_2$ is a graph $   (\Z^2 , E_2)$
with set of edges
$$
E_2 = \left \{  \{  ( x_1, y_1) , (x_2, y_2)    \} :    x_1, y_1 , x_2, y_2 \in \Z,
\,\,\,|x_1    -x_2  |+|  y_1-y_2    | =1 \right   \}    .
$$
In the sequel we will define \emph{horizontal} edges to be those with $|x_1- x_2| =1$ and \emph{ vertical} edges those with $|y_1-y_2| =1$.
We will also use the oriented square lattice $\vec{G}_2 =( \Z^2, \vec{E}_2 )$, where the set
of oriented edges is
\begin{equation}\label{e2}
\vec{E}_2 =\left  \{     \left  (  (x,y) , (x+1,y) \right ): x, y \in \Z  \right   \} \cup
\left  \{     \left  (  (x,y) , (x,y+1) \right ): x, y \in \Z  \right   \}  .
\end{equation}
In other words each horizontal edge in $ E_2 $ is oriented to the right, and each vertical edge is oriented upward.
The origin   is denoted by $O =(0,0)$.

A finite path $\gamma$ from $v_I\in \Z^2 $ to
$v_F\in \Z^2 $ is a finite sequence of vertices and
edges
$$
\gamma=( v_I=v_0, e_1,v_1,  e_2, \ldots,v_{N-1},  e_N , v_N =v_F)
$$
with $v_k \in \Z^2$ and
 $e_k =\{v_{k-1}, v_k\}\in E_2$, for $k =1, \ldots , N $. 
 The number $N$ of edges used by the path $\gamma$ is denoted by $|\gamma|$. In the sequel we will say that the path $\gamma $ joins $v_I $ to $v_F$.
    All the paths  from $v_I  $  to $v_F $ are collected in the set $\Gamma(v_I,v_F)$. Paths from $v_I$ to $v_F$ can be specified either by the sequence of vertices or by the sequence of edges.
 Paths with $v_I=v_F$ are called cycles.  Finally, we call a path \emph{oriented} if $(v_{k-1}, v_k) \in \vec{E} $ for $k =1, \ldots, N$.

   An infinite path $\gamma$ from $v_I $ is an infinite sequence of
vertices and edges
$$
\gamma=( v_I= v_0, e_1,v_1,  e_2, \ldots,v_{N-1},  e_N
, \ldots )
$$
such that $v_k \in \Z^2$ and
 $e_k =\{v_{k-1}, v_k\}\in E_2$, for $k \in \N $. Oriented infinite paths are defined likewise.
  An infinite \emph{self-avoiding path from} $v_I$ is an infinite path starting in $v_I$
 with  all the vertices different. Oriented paths are always self-avoiding. For any $\mathbf{b} \in \Z^2$, we denote by $\Gamma_{\mathbf{b}}$ the collection of infinite self-avoiding paths starting from $\mathbf{b}$.

The translation $\gamma+\mathbf {b}$ of a path $\gamma$ (finite or infinite) by means of a vector $\mathbf {b} \in \Z^2$ is defined by translating all the vertices by $\mathbf {b}$, inserting the appropriate edges in between. When the path $\gamma_1$ ends in the vertex where $\gamma_2$ starts, the two paths can be \emph{concatenated}, giving rise to a new path indicated by $\gamma_1 \odot \gamma_2$: vertices in $\gamma_1$ are followed by vertices in $\gamma_2$, except the first one. Notice that the concatenation of two self-avoiding paths
is not necessarily self-avoiding. A cycle $\sigma$ can be concatenated with itself an arbitrary number $i$ of times: the resulting cycle will be indicated by $\sigma^{\odot i}$.
Another operation on a finite path which is worth to introduce is its \emph{reversal}: the reversal of
$ \gamma=(v_0, e_1,v_1,  \ldots,v_{N-1},  e_N , v_N)$ is $-\gamma=(v_N, e_N
,v_{N-1},  \ldots,  v_{1},  e_1 , v_0)$.
Finally, given a finite path $\gamma=( v_0, e_1,v_1,  e_2, \ldots,v_{N-1},  e_N, v_N ),$
we find convenient to use the notation $\gamma_{a,b}$ for the \emph {truncation} $(v_a,e_{a+1},v_{a+1},\ldots,v_{b-1},e_b, v_{b})$, where $0\leq a <b \leq N$.

\medskip

Now let $ \mathbf{X} =(X_e: e \in E_2) $ and $ \mathbf{Y}  =(Y_e: e \in E_2) $   be two sequences of independent random variables, where
\begin{itemize}
 \item[a)]   $ \mathbf{X}  $ and $    \mathbf{Y}  $ are independent;
  \item[b)] for any $e \in E_2$, $X_e$ has the same law  $\mathcal{L}$ which   is different from $\delta_0$ to avoid trivialities;
  \item[c)] if $e =\{ (x,y), (x+1,y) \}$ then $\mathbb{P} (Y_e=1 ) =p_o$ and $\mathbb{P} (Y_e=-1 ) =1-p_o$;
  \item[d)]   if $e =\{ (x,y), (x,y+1) \}$ then $\mathbb{P} (Y_e=1 ) =p_v$ and $\mathbb{P} (Y_e=-1 ) =1-p_v$.
\end{itemize}
Without loss of generality we can assume  $  \frac{1}{2} \leq p_o \leq p_v \leq 1$, by changing the orientation of the axes
and exchanging the two coordinates when needed.

For a finite path $\gamma = (v_0,e_1 , v_1, e_2 , \ldots, e_{|\gamma|}, v_{|\gamma|})$ or an infinite one $\gamma = (v_0,e_1 , v_1, e_2 , \ldots)$, we define
\begin{equation} \label{sparziali}
T_N (\gamma)=\sum_{k=1}^{N}
 Z_{ (v_{k-1}, v_{k})   } ,\,\,\, S_N (\gamma)=\sum_{k=1}^N  Z_{ (v_{k-1}, v_{k})   } X_{e_k}, \text{  } N \leq |\gamma| \text{ or } N<\infty
\end{equation}
where
\begin{equation}\label{ZZtop}
  Z_{ (v_{k-1}, v_{k})   } =     (x_{k+1} - x_{k}  + y_{k+1}   - y_{k}  ) Y_{e_k}  ,
\end{equation}
 for $v_{k} =(x_k , y_k) \in \Z^2 $,  $k =1,  \ldots, N $.

 Notice that $  (x_{k+1} - x_{k}  + y_{k+1}   - y_{k}  ) $ is either $+1$, when the edge
$e_k $ is crossed according to the orientation of $\vec {E}_2$, or $-1$, when it is crossed in the opposite direction.
Therefore it is legitimate to interpret the field $(Y_e: e \in E_2)$ as defining a random orientation of $E_2$ with the following prescription: each path crossing an edge $e $ in agreement (in opposition) with this orientation receives a contribution to its sum equal to $X_e$ ($-X_e$).

For finite paths $\gamma$ of length $N$, we will preferably  write $T(\gamma)=T_N (\gamma)$ and $S(\gamma)=S_N (\gamma)$.

\medskip

Now let us define the random variables
$$
\Sigma (\mathbf {b})=\inf_{\gamma \in \Gamma_{\mathbf {b}}} \sup_{N \in \mathbb {N}} |S_N (\gamma)|, \text{ for } \mathbf{b} \in \Z^2.
$$
By translation invariance it is clear that the law of $\Sigma (\mathbf {b})$ does not depend on $\mathbf {b}$. $\Sigma (\mathbf {b}) $ is either infinite almost surely or its distribution function is positive at some positive value.
We denote
by $M_c = M_c (p_o,p_v ; \mathcal{L})$  the infimum of these values or $+ \infty$ when no such value exists. Here is our main result.

  \begin{theorem}\label{primo}
  For any $p_o, p_v \in [0,1]$ and for any law $\mathcal{L}$ the constant $ M_c (p_o,p_v ; \mathcal{L})    $ is finite.
  \end{theorem}

The finiteness of  $M_c$  means that for any $\delta > 0 $ there is a positive probability of finding
a self-avoiding path from a given $\mathbf{b}$ whose partial sums are bounded by $M_c + \delta$.

The following is an easy consequence of the previous theorem.
\begin{corollary} \label{coro}
For any $p_o, p_v \in [0,1]$ and for any law $\mathcal{L}$,
\begin{itemize}
\item[\emph{i.}] for any $\delta >0$ there exists $\mathbf{b} \in Z^2$ and a self-avoiding  path $\gamma \in \Gamma_{\mathbf{b}}$ with partial sums bounded by $M_c + \delta$;
\item[\emph{ii.}]  the random variable $\Sigma (\mathbf{b} )$ is almost surely finite, for any $\mathbf{b}\in \Z^2$.
\end{itemize}
\end{corollary}
\begin{proof} The proof of item i. is an immediate consequence of the ergodicity of the model.
As far as item ii.  is concerned, let $\gamma$ be a path as in item i. with say $\delta=1$.  One can always  construct a path from $\mathbf{b}$
with bounded partial sums by first following an arbitrary path starting from $\mathbf{b}$ which intersects $\gamma$. After the first intersection the path $\gamma$ is followed.
\end{proof}

With the next result we turn our attention to the recurrence properties of partial sums of infinite paths.

\begin{theorem} \label{Markovtorna}
For any $p_o, p_v \in (0,1)$, there exists almost surely a   self-avoiding path  $\hat \eta \in \Gamma_O$ such that the sequence of partial sums $(S_n (\hat \eta ) : n \in \N)$ is bounded and has zero as an accumulation point.
\end{theorem}

When proving Theorem \ref{Markovtorna} we will also clarify how in some cases its statement can be
strengthened.

The following  proposition is a minor addition to  Theorem \ref{primo}. It describes
the situations in which $M_c =0$.

 \begin{prop}\label{secondo}
 For any $p_o, p_v \in [0,1]$  and for any  law
 $\mathcal{L}$,
 \begin{itemize}
 \item[a.]  if $\mathbb{P} (X_e =0 ) < \frac{1}{2}$ then
    $ M_c (p_o,p_v ; \mathcal{L})  >0 $;
\item[b.]   if $\mathbb{P} (X_e =0 ) > \frac{1}{2}$ then
        $ M_c (p_o,p_v ; \mathcal{L}) =0  $.
 \end{itemize}
  \end{prop}

Another random variable we will be interested in is the following
$$
\bar {\Sigma} =\sup_{u,v \in \mathbb {Z}^2} \inf_{\gamma \in {\Gamma}_{u,v}} \sup_{N \leq |\gamma|} |S_N (\gamma)| .
$$
 Actually,  $ \bar {\Sigma} $  is almost surely constant because it is
a random variable that is invariant  with respect to the translations of an  ergodic system. The almost sure value of $\bar {\Sigma}$ will be denoted  by
$\bar M_c = \bar{M}_c (p_o,p_v ; \mathcal{L}) $.

Notice that $\bar M_c  = + \infty  $ means that no matter how large the constant $C>0$ is,
there exist $u, v \in \Z^2 $ with the property that any path, not necessarily self-avoiding, from $u$ to $v$ has a partial
sum that exceeds $C$ almost surely. The necessary and sufficient conditions ensuring $\bar M_c < + \infty $ are rather restrictive,
as stated in the following theorem
\begin{theorem} \label{uniforme}
For $p_o, p_v \in (0,1)$ then
$$
\bar M_c (p_o, p_v, \mathcal{L}) < \infty  \Leftrightarrow
{\mathcal{L}} \text{ has bounded support and no atoms different from zero.  }
$$
\end{theorem}

For $p_o, p_v \in (0,1)$, comparing Theorem \ref{primo} with Theorem \ref{uniforme}, one can see that the self-avoiding path, whose existence is ensured by the former theorem, has to avoid some ``bad" random subregions of
$\Z^2$, at least if $\mathcal {L}$ has unbounded support or it has atoms different from zero.

\section{Preliminary lemmas}

This section is devoted to establish some general results concerning sums of independent
 random variables which will be fundamental in the following. Since they could have independent interest we present them in
a more general context than needed.

\begin{lemma}\label{CLT}
Let $\phi: \N \to \mathbb{R}_+$ be a function with the property $\phi (N)=o( \sqrt{N})$, as $N \to \infty$.
Let  $(X_n : n \in \N)$ be a sequence of i.i.d. random variables, and let $(Z_n   : n \in \N)  $
be an independent  sequence of independent sign variables with $r_n =\mathbb{P} ( Z_n =+1 ) $.
If
\begin{itemize}

  \item[a.]  the random variable $X_1$  is  a.s. equal to a non zero constant
   and
   \begin{equation}\label{varia}
    \liminf_{N\to \infty} \frac{1}{N}\sum_{k=1}^N r_k (1-r_k )  >0,
\end{equation}
\end{itemize}
or
\begin{itemize}
\medskip

   \item[b.] the random variable $X_1$ is not a.s. constant,

\end{itemize}
 then
\begin{equation} \label{bo}
\lim_{N\to \infty}     \mathbb{P} (  |  \sum_{k=1}^{N} Z_k X_k | > \phi_N) =1.
\end{equation}
\end{lemma}
\begin{proof}
a. Suppose $X_1$ is a non zero constant a.s.. Since \eqref{varia} holds we can
apply the Lyapunov central limit theorem to $\sum  _{k=1}^{N} Z_k$ which leads to \eqref{bo}.

\noindent
b. Let us choose 
a cutoff $K >0 $ large enough to ensure  that
 $\pi_K : = \mathbb{P}  (  |   X_1 | <  K) > \frac{1}{2}$ and $ \sigma_K^2 :=Var (  X_1 \mathbf{1}_{ \{  |X_1| <K  \}  } ) >0$.
 Let us define the random set of indices corresponding to the $X_i$'s which exceed the cutoff
 $$
 I_N = \{ i \leq N : |X_i | \geq K       \}
 $$
so that the random variable $|I_N|$ has the binomial distribution $Bin (N, 1- \pi_K)$.  Finally define $R_N:=\sum_{k\in I_N }       Z_k   X_k$ and $G_N:=\sum_{k\notin I_N, k \leq N }   Z_k        X_k$. Now notice that
$$
 \mathbb{P} (  |  \sum_{k=1}^{N}    Z_k        X_k | \leq  \phi_N) =
 \mathbb{P} (  |  R_N+G_N   | \leq  \phi_N) \leq
$$
$$
\leq \mathbb{E} (\mathbb{P} (  |  R_N+G_N   | \leq  \phi_N     |   I_N, R_N )1_{\{|I_N|\leq \frac {2N}{3}  \}}) +\mathbb{P}  (|I_N| > 2N/3) \leq
$$
\begin{equation} \label{insup}
\leq \sup  \left \{      \mathbb{P} (  |  c+G_N  | \leq  \phi_N     |   I_N = I )
  : c\in \R, 0\leq |I| \leq 2N/3   \right  \} +   \mathbb{P}  (|I_N| > 2N /3).
\end{equation}
The last inequality is a consequence of the fact that  the random variables $R_N$ and $G_N$ are independent, conditionally to $I_N$.

Using Chernoff's theorem (see e.g. \cite{denHollander}) for the sequence $|I_N| $ one has that there exists a  positive constant $\lambda >0 $ such that
$ \mathbb{P}  (|I_N| > 2N /3) \leq \exp{  ( -\lambda   N) }$,  for any $N \in \mathbb{N}$.

As far as the first summand in \eqref{insup} is concerned, if we replace the distribution of $G_N$ conditional to $I_N$ with a Gaussian one with same mean and variance, we can bound the error by using the Berry-Esseen inequality (see \cite{Berry,Esseen}). In the Gaussian term the supremum w.r.t. $c$ is achieved by $c=-\mathbb {E}(G_N)$. Moreover it is easy to obtain that
$$
Var (G_N|I_N=I) \geq  (N- |I| ) \sigma_K^2 \geq \frac {N}{3}\sigma_K^2
$$
as long as $|I|\leq \frac {2N}{3}$, irrespectively of $(r_n)_{n \in \N}$. Moreover
$\mathbb{E}(| X_k \mathbf{1}_{\{ |X_k | <K \}} |^3) \leq K^3   $.

Altogether we obtain that the r.h.s. of \eqref{insup} can be bounded by
\begin{equation} \label{insup2}
     2\left [  \Phi \left (  \frac{\phi_N}{    \sigma_K  \sqrt{N  /3  }}      \right )  -
      \frac{1}{2}   \right ]     + \frac{   2   K^3     }{ \sigma_K^3 \sqrt{N/3} }   +
        \exp ( - \lambda N).
\end{equation}
We conclude the proof by observing that the three summands
in \eqref{insup2} go to zero when $N $ increases to infinity.
\end{proof}

Before stating the next lemma we need to recall the definition and some of the main properties of the total variation distance of two probability
measures $\mu$ and $\nu $ on the same measurable space $ (\Omega, \mathcal{F})$.

\begin{definition} \label{defTV}
The total variation distance between probability measures $\mu$ and $\nu $ on $ (\Omega, \mathcal{F})$ is defined as
\begin{equation} \label{definisceTV}
|| \mu - \nu||_{TV} = \sup_{A \in \mathcal{F}} \mu (A)     -\nu (A).
\end{equation}
\end{definition}

Here are the properties of total variation we are going to use in the sequel (see e.g. \cite{Lindvall}):

\begin{itemize}
\item[i.]
\begin{equation} \label{definisceTVi}
|| \mu - \nu||_{TV} = \sup_{0\leq f \leq 1} \mathbb {E}(f(X))-\mathbb {E}(f(Y)),
\end{equation}
where $f$ is a measurable function on $ (\Omega, \mathcal{F})$, and $X$ and $Y$ are random variables with laws $\mu$ and $\nu$, respectively;
\item[ii.] For $X$ and $Y$ random variables on the \emph {same} probability space, with laws $\mu$ and $\nu$, respectively
\begin{equation} \label{definisceTVii}
|| \mu - \nu||_{TV}\leq \mathbb {P}(X \neq Y)
\end{equation}
and the equality is achieved by some choice of $X$ and $Y$ (\emph {maximal coupling} of $\mu$ and $\nu$);
\item[iii.]
If $\mu \ll \lambda$, $\nu \ll \lambda$, then
\begin{equation} \label{definisceTViii}
|| \mu - \nu||_{TV}= \mu (\tilde A)-\nu (\tilde A)
\end{equation}
where
$$
\tilde A=\left \{\omega: \frac {d\mu}{d\lambda} (\omega) \geq \frac {d\nu}{d\lambda}(\omega) \right \}.
$$
\item[iv.] When $\Omega=\Z$, then
\begin{equation} \label{definisceTViv}
|| \mu - \nu||_{TV}=\frac {1}{2}\sum_{x \in \Z} |\mu(\{x\})-\nu(\{x\})|.
\end{equation}
\end{itemize}

The following form of the local central limit theorem will be of interest later.

 \begin{lemma} \label{TLL}
Let $(Z_i : i \in \N)$ be independent sign variables
with $r_i = \mathbb{P} (Z_i = +1)$, for $i \in \N$. Let
\begin{equation}\label{tienne}
 T_N=\sum_{i=1}^N Z_i,
\end{equation}

and
\begin{equation}\label{med-var}
a_N :=  \mathbb E (T_N) = \sum_{i=1}^N (2r_i -1) , \,\,\,     b_N^2 :=   \text {Var}(T_N)= 4\sum_{i=1}^N  r_i    (1- r_i )    .
\end{equation}
Suppose that
\begin{equation}\label{varianza}
 \lim_{N \to \infty}   \frac{a_N}{\sqrt{N}} =0, \,\,\,\,\, \liminf_{N \to \infty} \frac{ b_N^2}{N}  >0.
\end{equation}
Then
\begin{equation}\label{TV-gaussiana}
  \lim_{r \to \infty} || \mathcal{L}   (  T_N   ) -\mathcal{L} (T_{N,\sigma}) ||_{TV} =0  ,
\end{equation}
where $T_{N,\sigma}$ has the symmetric law
\begin{equation}\label{gauss}
  \mathbb{P} (T_{N,\sigma}=k)=\Phi \left( \frac {k+1}{b_N}\right)-\Phi \left( \frac {k-1}{b_N}\right ),
\end{equation}
for $k \in L_N = 2 \mathbb{Z} + (N \mod (2))$,  $ \Phi $ being the standard Gaussian distribution function.
\end{lemma}
\begin{proof}
First notice that the condition \eqref{varianza} implies the Lyapunov condition for the validity of
the CLT for the sequence $ (Z_i : i \in \N) $. Moreover, the hypotheses of  Theorem 1.1 in \cite{BM1995} hold implying that the local
central limit theorem holds true. This can be written in the convenient form  (see formula (1.3) in \cite{BM1995})
\begin{equation}\label{BM1.3}
\delta_N =b_N \sup_{k \in I_N}       |    \mathbb{P}  (   T_N   =k    )   -
    \mathbb{P} (\tilde T_N=k)     | =o(1),
\end{equation}
where  $\tilde T_N$ has the law
$$
\mathbb{P} (\tilde T_N=k)=\Phi \left( \frac {k+1-a_N}{b_N}\right)-\Phi \left( \frac {k-1-a_N}{b_N}\right ),
\text{ for } k \in L_N .
$$

Next
\begin{equation*}\label{primaserie}
|| \mathcal{L}   ( \tilde T_N   ) -\mathcal{L} (T_N) ||_{TV}=\sup_{A  \subset I_N } \, \left [\sum_{k \in A}
\left (\mathbb{P}(\tilde T_N=k)-\mathbb{P}(T_N=k)     \right )      \right ]
\end{equation*}
\begin{equation}\label{secondaserie}
  \leq \frac {1}{2} \sum_{k \in I_N : |k -a_N | \leq \frac{b_N}{ \sqrt{\delta_N} }}
\left |\mathbb{P}(\tilde T_N=k)-\mathbb{P}(T_N=k)     \right|    + \frac {1}{2} \sum_{k \in I_N : |k -a_N | > \frac{b_N}{ \sqrt{\delta_N} }}
\mathbb{P}(\tilde T_N=k)
\end{equation}
\begin{equation}\label{terzaserie}
  \leq C \left (\sqrt {\delta_N} + 2\Phi \left (-\frac{1}{\sqrt{ \delta_N}} \right    ) \right ),
\end{equation}
for some constant $C$, which goes to zero when $N$ goes to infinity (notice that in the next to the last inequality we used property iv. of the total variation distance).
It remains to prove that
\begin{equation}\label{vargauss}
\lim_{N \to \infty} ||\mathcal{L}(\tilde T_N)-\mathcal{L}(T_{N,\sigma})||_{TV}=0.
\end{equation}
The total variation distance in the above display can be upper bounded by the total variation distance between the Gaussian distributions $N(a_N,b_N^2)$ and $N(0,b_N^2)$, which, using property i. of the total variation distance, is clearly equal to that between their scale multiples $N(\frac {a_N}{b_N},1)$ and $N(0,1)$. Using property iii., the latter can be bounded from above by
$$
\frac {1}{2}\left [\Phi \left (\frac {1}{2} \frac {|a_N|}{b_N}\right)-\Phi \left (-\frac {1}{2}  \frac {|a_N|}{b_N} \right )\right ]     .
$$
Since the assumptions \eqref{varianza} clearly imply that $a_N/b_N$ tends to $0$ as $N \to \infty$, this proves the desired relation \eqref{vargauss}.
\end{proof}

\section{Tessellations of the euclidean plane}\label{problema}

We will construct a tessellation $\calT (a_1, a_2 )$ of the Cartesian plane $\mathbb{R}^2$
depending on two integer parameters $a_1\geq a_2 \geq 0$, with $a_1>0$.

The tessellation will be  obtained by translations of the basic parallelogram $R_{a_1,a_2}(O)$ with vertices
\begin{equation}\label{4punti}
  A_1=(a_1,-a_2),\,\,A_2=(a_1,-a_2+3m+1),\,\, A_3=(-a_1,a_2+3m+1),\,\,A_4=(-a_1,a_2),
\end{equation}
where $m$ is a positive integer to be suitably chosen. It is immediately verified that the $y$-axis cuts
the parallelogram into two equal sides $R_{a_1,a_2}^l(O)$ and $R_{a_1,a_2}^r(O)$ (the left and the right parallelogram, respectively). The tessellation is then defined as
\begin{equation}\label{sitiblocchi}
R_{a_1,a_2} (\mathbf {b})    =        R_{a_1,a_2} (O) + b_x A_2 +  b_y A_3,
\end{equation}
for $\mathbf {b}=(b_x, b_y) \in \Z^2$. As observed before each $R_{a_1,a_2} (\mathbf {b})$ is cut into the two equal sides $R_{a_1,a_2}^l (\mathbf {b}) $ and $R_{a_1,a_2}^r (\mathbf {b}) $, obtained by translating $R_{a_1,a_2}^l (O) $ and $R_{a_1,a_2}^r (O) $ with the vector $b_xA_2+b_yA_3$, respectively.

For the sequel we need to
define an oriented graph structure on the tessellation $\calT (a_1, a_2 )$, by putting oriented edges from each parallelogram $R_{a_1,a_2}(\mathbf {b})$, with $\mathbf {b}=(b_x,b_y)$ to the parallelograms $R_{a_1,a_2}(b_x+1,b_y)$ and $R_{a_1,a_2}(b_x,b_y+1)$.  This structure is clearly isomorphic to $\vec {G}_2=(\mathbb {Z}_2, \vec {E}_2)$, where $\vec {E}_2$ is defined in \eqref{e2}. It is useful to associate to these edges the parallelograms $R_{a_1,a_2}^r(\mathbf {b})$ and $R_{a_1,a_2}^l(\mathbf {b})$, respectively (see Figure 1).

\begin{figure}[ht]
\centering
\includegraphics[width=90mm]{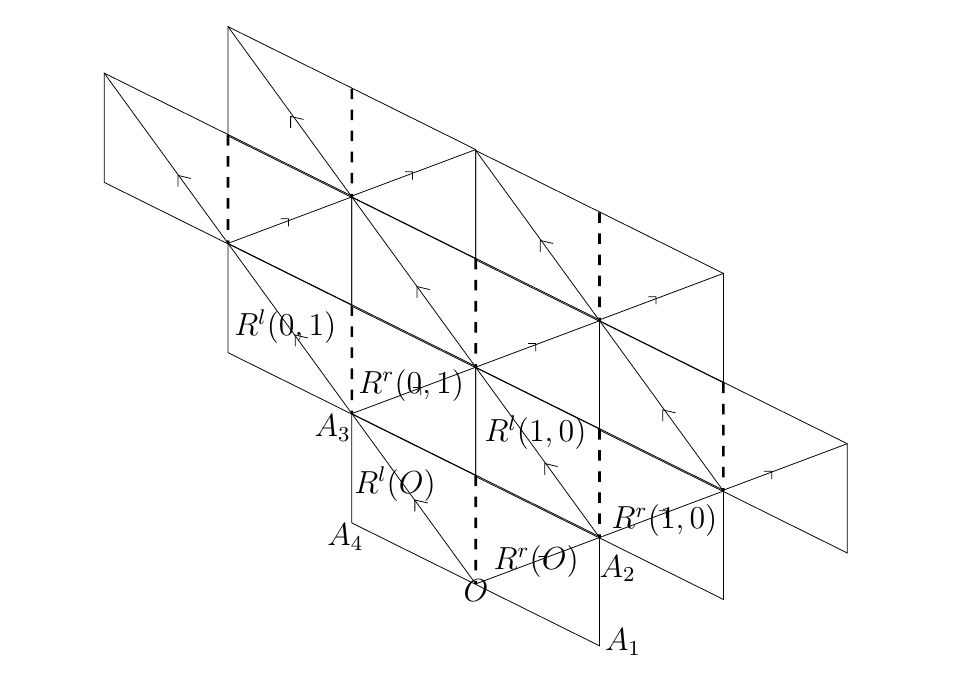}
 \caption{\small{Tessellation and oriented graph structure.
   }}
\end{figure}

Recall that $\frac{1}{2} \leq p_o \leq p_v \leq 1 $.
For $p_v \neq \frac {1}{2}$ define
 \begin{equation}\label{vengono}
   \rho  =  \frac{2p_o-1}{2p_v-1}    \in [0,1]           .
 \end{equation}
When $p_o=p_v = \frac{1}{2}$ set $\rho=0$. Now, to each value of $\rho$ we associate a \emph{ sequence} of tessellations $\mathcal{T}(a_1(n), a_2(n) )$, $n \in \mathbb{N}$.

When $\rho \in \mathbb {Q}_+$ we set $a_1(n)=na_1^*$ and $a_2(n)=na_2^*$, where $a_1^*$ and $a_2^*$ are coprime such that $\rho=\frac {a_2^*}{a_1^*}$. We extend this definition for $\rho=0$ setting in this case $a_1^*=1$ and $a_2^*=0$.

When $\rho \notin \mathbb {Q}$ we take increasing sequences of positive integers $a_1(n)$ and $a_2(n)$ with the property
\begin{equation}\label{acappa}
  \left |\frac {a_2(n)}{a_1(n)}-\rho\right | \leq \frac {1}{a_1(n)^2},
\end{equation}
as ensured by Dirichlet's approximation theorem (see e.g. \cite{Lang}).

\medskip

For $n \in \mathbb{N}$, $u \in \{r,l\}$ and $\mathbf{b}  \in \Z^2$, we denote by  $B^u_{a_1(n),a_2(n)} (  \mathbf{b}   )$  the set of edges (seen as closed segments) belonging entirely to the parallelogram $R^u_{a_1(n),a_2(n)}(  \mathbf{b}   )$. The family of \emph{blocks}
$$\{  B^u_{a_1(n),a_2(n)} (  \mathbf{b}   ) : u \in \{r,l\}, \,\,\mathbf{b}  \in \Z^2  \}
$$
is not a partition of $E_2 $, for two reasons. First, in general there exists  horizontal edges (again seen as closed segments) that do not lie entirely
in a parallelogram of the tessellation so they are excluded from any block. Second, there are vertical edges in common to two adjacent blocks.
We denote by $ \text{int} B^u_{a_1(n),a_2(n)}(  \mathbf{b}   )$ the set of edges which, seen as open segments, are subsets of the interior of  $R^u_{a_1(n),a_2(n)}(  \mathbf{b}   )$. For $u \in \{r,l\}$ and $\mathbf{b} \in \Z^2$, these sets are pairwise disjoint.

 Let us consider a path $\gamma_0^r(n)=( v_0, e_1,v_1,  e_2, \ldots,v_{\ell(n)-1},  e_{\ell(n)} , v_{\ell(n)} )$, from $v_0=O$ to $v_{\ell(n)}=A_1(n)=(a_1(n),-a_2(n))$ of $\Z^2$, with the following properties:
\begin{itemize}
  \item $\gamma_0^r(n)$ is \emph {decreasing}, in the sense that horizontal edges are crossed to the right and vertical edges downward (therefore its length $\ell(n)=a_1(n)+a_2(n)$);
  \item $v_1=(1,0)$ and $v_{\ell(n)-1}=(a_1(n)-1,-a_2(n))$ (therefore the first and the last edge of $\gamma_0^r(n)$ are horizontal);
  \item the entire path lies within the stripe $|y+\frac {a_2(n)}{a_1(n)}x|\leq 1$ in the cartesian plane with coordinates $x$ and $y$.
\end{itemize}

We are going to apply Lemma \ref{TLL} to the random variable $  T(\gamma_0^r(n))$. This random variable has the same law of
\begin{equation}\label{stessalegge}
  \sum_{i=1}^{a_1  (n)  } Z_i'    - \sum_{i=1}^{a_2  (n)  } Z''_i
\end{equation}
where $ (Z'_i )_{i \in \N} $ and $ (Z''_i )_{i \in \N} $ are two independent i.i.d. sequences of sign variables, with
\begin{equation}\label{bernu}
  \mP ( Z'_i = +1 ) = p_o , \,\,\,   \mP ( Z''_i = +1 ) = p_v  .
\end{equation}
As a consequence, when $n$ changes, since $ a_1(n)$ and $a_2 (n)$ are both increasing in $n$, it is possible to regard $T(\gamma_0^r(n))$ as a (sub)-sequence (of a sequence) of the form \eqref{tienne}. Now
\begin{equation}\label{mediapiccola}
\mathbb {E} (T(\gamma_0^r(n)))=a_1(n)(2p_o-1)-a_2(n)(2p_v-1)=a_1(n)(2p_v-1)\left (\rho-\frac {a_2(n)}{a_1(n)}\right )
\end{equation}
by \eqref{stessalegge}.
For $\rho \in \mathbb {Q}$ this is always equal to zero. For $\rho \notin \mathbb {Q}$ instead, using \eqref{acappa}, we get
\begin{equation}\label{boundmean}
\left |\mathbb {E} (T(\gamma_0^r(n)))\right | \leq \frac {1}{a_1(n)}\to 0
\end{equation}
as $n \to \infty$.   
 Indeed, the choice of the basic parallelogram and of the path $\gamma_0^r(n)$ is made to ensure this 
kind of  ``asymptotic unbiasedness". Moreover
\begin{equation}\label{varianzagrande}
\text {Var} (T(\gamma_0^r(n)))=4[a_1(n)p_0(1-p_0)+a_2(n)p_v(1-p_v)]=O(a_1(n)),
\end{equation}
unless $p_o=1$ (in which case also $p_v=1$, and $\text {Var} (T(\gamma_0^r(n)))=0$).

For $p_o<1$, Lemma \ref{TLL} and the property ii. of total variation
justifies the following maximal coupling construction: an auxiliary random variable $T_{\sigma} (\gamma_0^r(n))$ can be introduced, with a \emph {symmetric} law (recall \eqref{gauss}), such that the event
\begin{equation} \label{couplingevent}
H_0^r(n)=\{T_{\sigma}(\gamma_0^r(n))=T(\gamma_0^r(n))\}
\end{equation}
is realized with a probability which tends to $1$ as $n \to \infty$. For $p_o=p_v=1$, being $a_1(n)=a_2(n)=n$, $T(\gamma_0^r(n))$ has already a symmetric law, so one can take $T_{\sigma}(\gamma_0^r(n))=T(\gamma_0^r(n))$, in which case $H_0^r(n)$ is the entire sample space.

Next observe that the random variable
\begin{equation} \label{contador}
Q(\gamma_0^r(n))=\frac {l(n)+T(\gamma_0^r(n))}{2}
\end{equation}
counts the number of $+$ signs along the path $\gamma_0^r(n)$. Moreover the sum  $S(\gamma_0^r(n))$ of the path $\gamma_0^r(n)$ can be expressed, preserving the law, as
\begin{equation}\label{enfer}
S(\gamma_0^r(n))=\sum_{i=1}^{Q(\gamma_0^r(n))} X_{e_i}-\sum_{i=Q(\gamma_0^r(n))+1}^{l(n)}X_{e_i}=:f_n(T(\gamma_0^r(n));X_{e_1},\ldots,X_{e_{\ell(n)}})
\end{equation}
where $\{e_1,\ldots,e_{l(n)}\}$ are the edges of $\gamma_0^r(n)$. On the event $H_0^r(n)$ this sum coincides with
\begin{equation}\label{enfer2}
S_{\sigma}(\gamma_0^r(n))=f_n(T_{\sigma}(\gamma_0^r(n));X_{e_1},\ldots,X_{e_{\ell(n)}}),
\end{equation}
which is immediately verified to have a symmetric law.

Next we are going to define suitable vertical translations of the path $\gamma_0^r(n)$, namely
\begin{equation}\label{adestra}
\gamma_i^r(n)=\gamma_0^r(n)+(2+3(i-1))(0,1),\,\,\, i=1,\ldots,m.
\end{equation}
All these paths run from the ``left vertical" boundary to the ``right vertical" (see Figure 2) boundary of the parallelogram $R^r_{a_1(n),a_2(n)}(O)$, using only edges belonging to int$B^r_{a_1(n),a_2(n)}(O)$; moreover they are disjoint by construction, which implies that the sums along each of them are independent. 

\begin{figure}[ht]
\centering
\includegraphics[width=90mm]{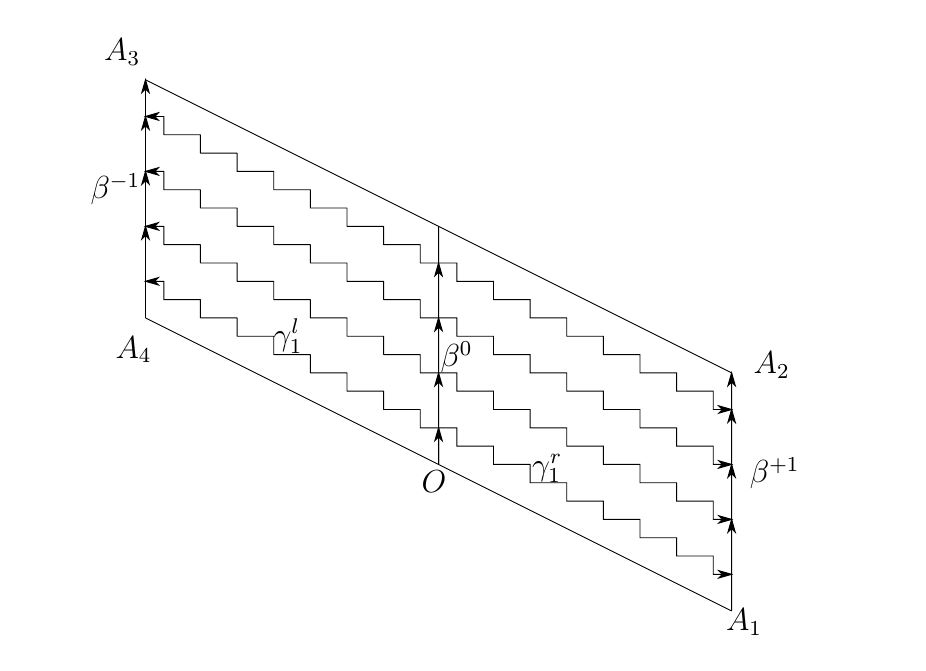}
   \caption{\small{The paths $\gamma^r_i$ and $\gamma^l_i $, for $i=1,\ldots , 4$, $\beta^0, \beta^{\pm 1}$.
   }}
\end{figure}

For any of these paths we repeat the same construction made for $i=0$, getting i.i.d.
\begin{equation}\label{camini}
T(\gamma_i^r(n)), S(\gamma_i^r(n)), T_{\sigma}(\gamma_i^r(n)), S_{\sigma}(\gamma_i^r(n)), 
\mathbf{1}_{H_i^r(n)} ,  \,\,\, i=1,\ldots,m,
\end{equation}
where 
$$
H_i^r(n)=\{T_{\sigma}(\gamma_i^r(n))=T(\gamma_i^r(n))\},
$$ 
that we  call the ``symmetry event" for the path $ \gamma_i^r(n) $  inside the  block $ B^r_{a_1(n),a_2(n)}(O)$.

For later use we need some book-keeping about the minimum and the maximum among the sums $S(\gamma_i^r(n))$'s. So let us define
\begin{equation}\label{indicer}
  i_r = \min \{ i =1, \ldots ,  m: S(\gamma_i^r(n)) =\min_{j=1,\ldots,m} S(\gamma_j^r(n))\},
\end{equation}
\begin{equation}\label{indicer2}
  j_r = \min \{      i =1, \ldots ,  m: S(\gamma_i^r(n)) =\max_{j=1,\ldots,m} S(\gamma_j^r(n))\}.
\end{equation}
and 
\begin{equation}\label{contributo?}
\gamma^r_-(n)=\gamma_{i_r}^r(n) , \,\,\,   \gamma^r_+ (n) =  \gamma_{j_r}^r(n),
\end{equation}   
that we call the ``minimum path" and the ``maximum path" in the block $B^r_{a_1(n),a_2(n)}(O)$, respectively.
\medskip

 Now we define the path $\gamma_0^l(n)$, starting in the origin $O$ and ending in $A_4(n)=(-a_1(n),a_2(n))$, obtained by reversing $\gamma_0^r(n)$ and translating it by $A_4(n)$. Repeating the constructing made before, we define the random variables
$T(\gamma_0^l(n))$, $S(\gamma_0^l(n))$, $T_{\sigma}(\gamma_0^l(n))$, $S_{\sigma}(\gamma_0^l(n))$,
and the event
 $$
H_0^l(n)=\{T_{\sigma}(\gamma_0^l(n))=T(\gamma_0^l(n))\}.
$$

 Observe that
 \begin{equation*}
\left ( T(\gamma_0^l(n)),S(\gamma_0^l(n)),T _{\sigma}(\gamma_0^l(n)),S_{\sigma}(\gamma_0^l(n)),H_0^l(n)    \right  )\Legge
 \end{equation*}
 \begin{equation}\label{ugualiinlegge}
   \Legge \left ( -T(\gamma_0^r(n)),-S(\gamma_0^r(n)), T_{\sigma}(\gamma_0^r(n)),S_{\sigma}(\gamma_0^r(n)),H_0^r(n)        \right  )        .
 \end{equation}

 \medskip

 Translating the path $\gamma_0^l(n)$ vertically we obtain the family
\begin{equation}\label{asinistra}
\gamma_i^l(n)=\gamma_0^l(n)+(2+3(i-1))(0,1)=-\gamma_i^r(n)+A_4(n),\,\,\, i=1,\ldots,m,
\end{equation}
of paths running from the ``right vertical" boundary to the    ``left vertical" boundary of the parallelogram $R^l_{a_1(n),a_2(n)}(O)$ (see Figure 2). Independently of the random variables constructed for the right block $B_{a_1(n), a_2 (n)}^r(O)$, we construct, with the same procedure, the ones for the left block $B_{a_1(n), a_2 (n)}^l(O)$  getting i.i.d.
\begin{equation}\label{camini2}
T(\gamma_i^l(n)), S(\gamma_i^l(n)), T_{\sigma}(\gamma_i^l(n)), S_{\sigma}(\gamma_i^l(n)),H_i^l(n),\,\,\, i=1,\ldots,m.
\end{equation}

 The indices $ i_l$ and $j_l $ and the paths $\gamma^l_-(n)$ and $\gamma^l_-(n)$ are defined analogously to \eqref{indicer}, \eqref{indicer2} and \eqref{contributo?}.

In order to construct a convenient collection of paths within the blocks $B_{a_1(n),a_2(n)}^r(O)$ and $B_{a_1(n),a_2(n)}^l(O)$ we need also to introduce the vertical paths (defined by the sequence of vertices)
$$
\beta^{h}=(hA_1(n)+i(0,1),\,\,\,i=0,\ldots,3m+1),\,\,\, h=-1,0,+1.
$$
These paths run along the left and right vertical boundary of $R_{a_1(n),a_2(n)}^r(O)$ (for $h=0$ and $h=+1$, respectively) and $R_{a_1(n),a_2(n)}^l(O)$ (for $h=-1$ and $h=0$, respectively), which clearly share a side
(see Figure 2). In order to simplify the notation we choose not to make explicit the dependence of these paths from $m$ and $n$. All the edges of the path $\beta^{0}$ are common to both blocks $B_{a_1(n),a_2(n)}^r(O)$ and $B_{a_1(n),a_2(n)}^l(O)$ (but they do not belong to their interiors). We call $E^{r,l}(m,n)=E^{l,r}(m,n)$, $E^{r,r}(m,n)$ and $E^{l,l}(m,n)$ the set of edges belonging to the paths $\beta^0$, $\beta^{+1}$, and $\beta^{-1}$, respectively.
Observe that in the notation $E^{u_1,u_2}(m,n)$, with $u_1, u_2 \in \{l,r\}$, the index $u_1$ indicates if the edges live in a right or left block, whereas the second indicates if they belong to the left or right    ``vertical boundary" of such a block.
 Notice that
\begin{equation}\label{quattroemme}
\left | E^{u_1,u_2}(m,n) \right |=3m+1, u_i \in \{r,l\}, i=1, 2.
\end{equation}

Once all these paths have been defined we can build by suitable concatenations two families of paths $(\eta_i^r(n), i=0,\ldots,m-1)$ and $(\eta_i^l(n), i=1,\ldots,m)$, joining the origin $O$ with the vertices
$A_2(n)$ and $A_3(n)$, respectively. They use edges within the blocks $B_{a_1(n),a_2(n)}^r(O)$ and $B_{a_1(n),a_2(n)}^l(O)$, respectively, and are defined by 
\begin{equation*}
  \eta_i^r(n)=\beta^0_{0,2+3(i-1)}\odot \gamma_i^r(n)\odot \beta^{+1}_{2+3(i-1),3m+1},
\end{equation*}
\begin{equation}\label{catenacon}
  \eta_i^l(n)=\beta^0_{0,2+3(i-1)}\odot \gamma_i^l(n)\odot \beta^{-1}_{2+3(i-1),3m+1},
\end{equation}
for $i=1,\ldots ,m$. The path $\eta_i^r(n)$ ($\gamma_i^l(n)$) starts with vertical edges, along the $y$-axis,
until it reaches the initial vertex of the path $\gamma_i^r(n)$ ($\gamma_i^l(n)$, respectively), which is followed until the end.  Then a convenient number of vertical edges allows to reach  $A_2(n)$ ($A_3(n)$, respectively). It is clear that all these paths are self-avoiding. 
Along these paths we are able to control the corresponding sums $S(\eta_i^u(n))$, $u \in \{r,l\}$.

Indeed, for the sum along each path $\eta_i^u(n), i=1,\ldots, m$,  $u \in \{r,l\}$ the following holds
\begin{equation}\label{contributo}
  S(\eta_i^u(n))=S(\beta^0_{0,2+3(i-1)}  )    + S( \gamma_i^u(n)  )   +S( \beta^{\pm 1}_{2+3(i-1),3m+1}   )  ,
\end{equation}
where the three terms are independent (notice that in the third summand at the r.h.s. the sign is $+1$ when $u=r$ and it is $-1$ when $u=l$, see \eqref{catenacon}).
However notice that, whereas the random variables
$$
\left \{ S( \gamma_i^r(n)  ),\,\,\,  S( \gamma_i^l(n)  ) \,\,\, i=1,\ldots,m \right  \}
$$
are all independent, this is not true for the random variables
$$
\left \{ S( \eta_i^r(n)  ), \,\,\, S( \eta_i^l(n)  )      ,\,\,\, i=1,\ldots,m  \right  \}
$$
due to the presence of the first summand at the r.h.s. of \eqref{contributo}.
For later use define also
\begin{equation}\label{contributo2}
 \eta^u_-(n)=\eta_{i_u}^u(n) ,    \,\,\, \eta^u_+ (n) =  \eta_{j_u}^u(n) .
\end{equation}

Finally we define a number of events whose intersection will determine the \emph{goodness} of  a block  $B_{a_1(n),a_2(n)}^u(O)$, with $u \in \{r,l\}$.

For any $u \in \{ r,l\}$, positive integers $m$ and $n$, and positive constants $K_1$ and $K_2$,  define the events
\begin{equation}\label{F0}
  F_0^u(m,n)=\bigcap_{ i=1  }^{m} H_i^u(n),
\end{equation}
\begin{equation}\label{F1}
  F_1^u(m,n, K_1)=\left \{   \sum_{    e \in E^{u,r}(m,n)     }        |X_e|\leq \frac {K_1}{2}  \right \}\cap \left \{   \sum_{    e \in E^{u,l}(m,n)     }        |X_e|\leq \frac {K_1}{2}  \right \}
\end{equation}
\begin{equation}\label{F2}
 F_2^u (m,n, K_1) =\left \{  S(\gamma_{+}^u(n)) \geq 2K_1  \right \}
 \cap \left \{  S(\gamma_{-}^u(n)) \leq  -2K_1  \right \}  ,
\end{equation}
and
\begin{equation}\label{F3}
  F_3^u(m,n, K_2) =\{ |X_e | \leq K_2 : e \in \bigcup_{i =1}^{m} \gamma_{i}^u(n)  \} .
\end{equation}

The event $F_0^u(m,n)$ ensures that all the symmetry events for the paths $\gamma_{i}^u(n)$, $i=1,\ldots, n$ are realized. The  symmetrization of the variables $T(\gamma_{i}^u(n))$ and $S(\gamma_{i}^u(n))$ is essential to establish the forthcoming identity \eqref{orsa}. The realization of $F_1^u(m,n)$ and $F_2^u(m,n)$ guarantees the availability of a path with a desired sign within the block, whereas the realization of $F_3^u(m,n)$ allows to control the contribution of each individual term to the sums along the paths $\gamma_{i}^u(n)$.

We say that the block  $B^u_{a_1(n),a_2(n)} (O)$ is \emph{good} when the event
\begin{equation}\label{G}
G^u   (m,n, K_1, K_2)     =      F_0^u(m,n) \cap     F_1^u(m,n, K_1) \cap  F_2^u (m,n, K_1) \cap     F_3^u(m,n, K_2)
\end{equation}
is realized.

From \eqref{ugualiinlegge} one has that
\begin{equation}\label{orsazero}
  \mathbb{P} (  G^r   (m,n, K_1, K_2)     )  =
  \mathbb{P} (    G^l  (m,n, K_1, K_2)  ) .
\end{equation}
 If $F_0^u(m,n)$ is realized the random variables $S(\gamma_i^u(n))$, for $i=1,\ldots,m$, have the same law irrespectively of $u \in \{r,l\}$. In particular, this implies that
\begin{equation}\label{orsa}
(S(\eta^r_-(n)  ) ,S(\eta^r_+(n))  \,| \, G^r   (m,n, K_1, K_2)    \Legge   (S(\eta^l_-(n)  )   ,S(\eta^l_+(n))   \,   | \, G^l   (m,n, K_1, K_2).
\end{equation}

Because of \eqref{orsazero}, in the next lemma, without loss of generality, we are allowed to refer only to the right block $B^r_{a_1 (n) , a_2 (n) } (O)$. This lemma ensures that the goodness of a block can be obtained with a probability arbitary close to $1$, with a suitable choice of the parameters.

\begin{lemma} \label{lemmaT}
Suppose that either $p_o <1$ or at least $X_e $ is not a.s. constant. For any $\varepsilon >0 $
there exist $\bar m $, $\bar n$, $\bar K_1$  and  $\bar K_2 $ such that
\begin{equation}\label{probF}
         \mathbb{P} (   G^r   (  \bar  m,   \bar n, \bar  K_1,  \bar K_2)  ) \geq 1- \varepsilon  .
\end{equation}
\end{lemma}
\begin{proof}
Let us define
\begin{equation}\label{emme}
\bar m = \lfloor \log_2 \varepsilon^{-1} \rfloor +4.
\end{equation}
Once this choice is made, there exists $\bar K_1$ such that
\begin{equation}\label{numerouno}
\mathbb{P} ( F_1^r (  \bar m,n, \bar K_1) ) \geq 1 -\frac{\varepsilon}{4},
\end{equation}
for any $n$ (actually the left hand side does not depend on $n$). 

Let us explain how to choose $\bar n $ to guarantee both
\begin{equation}\label{numerodue2}
\mathbb{P} ( F_0^r (  \bar m,\bar n) ) \geq 1 -\frac{\varepsilon}{4} ,\,\,\,
\mathbb{P} ( F_2^r (  \bar m,\bar n, \bar K_1) ) \geq 1 -\frac{\varepsilon}{4}.
\end{equation}
First notice that, since
\begin{equation}\label{eventcoupling}
\mathbb {P}(F_0^r(\bar m, n))=\mathbb {P} (H_0^r(n)=1)^{\bar m},
\end{equation}
tends to $1$ as $n \to \infty$ (see \eqref{couplingevent}), the first inequality is obtained for $n $ large enough.

Next  define the events
$$
A^r_{+,i} (n)= \{   S(\gamma^r_i(n)) \geq 2 \bar K_1            \} , \,\,\, A^r_{-, i} (n) = \{   S(\gamma^r_i(n)) \leq  - 2 \bar K_1            \} , \text{ for } i = 1, \ldots , m,
$$
and observe that
$$
 F_2^r (  \bar    m,n,   \bar  K_1)   = \left ( \bigcup_{ i=0}^{\bar m-1}  A^r_{+,i} (n)  \right )  \cap \left ( \bigcup_{ i=0}^{\bar m-1}  A^r_{-,i} (n)  \right ),
$$
therefore
$$
\mathbb{P} (     F_2^r (  \bar    m,n,   \bar  K_1)     ) \geq  1-
( 1- \mathbb{P} (A^r_{+,1} (n) ))^{ \bar m } -
( 1- \mathbb{P} (A^r_{-,1} (n) ))^{ \bar m } .
$$
If we prove that
\begin{equation}\label{limiti}
\lim_{n \to \infty} \mathbb{P}(A^r_{+,1} (n))=\lim_{n \to \infty} \mathbb{P}(A^r_{-,1} (n))=1/2,
\end{equation}
then
$$
\liminf_{n \to \infty} \mathbb{P} (     F_2^r (  \bar    m,n,   \bar  K_1)     ) \geq 1- \left (\frac {1}{2}\right )^{\bar m -1}> 1- \frac {\varepsilon}{4},
$$
where the last inequality is guaranteed by the choice \eqref{emme}. As a consequence for $n $ large enough both inequalities in
\eqref{numerodue2} hold.

Finally, we split the proof of \eqref{limiti} in two cases.

\emph{Case 1}: $p_o<1$.

We apply Lemma \ref{CLT} and Lemma \ref{TLL} to the sequence of random variables
\begin{equation}\label{sumform}
S(\gamma^r_0(n))=\sum_{k=1}^{l(n)} Z_{(v_{k-1}(n), v_k(n))}X_{e_k(n)},
\end{equation}
where $\gamma^r_0(n)=(v_0(n), e_1(n), v_1(n),\ldots,e_{\ell(n)}(n), v_{\ell(n)}(n))$.
More precisely Lemma \ref{CLT} serves to ensure that $\mathbb{P}(A^r_{+,0} (n))+\mathbb{P}(A^r_{-,0} (n))$ tends to $1$ as $n \to \infty$. For proving that each of the terms go to $1/2$, recall that
we already established that \eqref{boundmean} and \eqref{varianzagrande} hold. By Lemma \ref{TLL}, the random variable
$   S(\gamma^r_0(n))    $ is equal to $   S_{\sigma}(\gamma^r_0(n))    $ on the  event $H_0^r (n)$ whose  probability tends to $1$ as $n \to \infty$. Since $   S_{\sigma}(\gamma^r_0(n))    $ has a symmetric law  this implies \eqref{limiti}, which ends the proof of \emph{Case 1}.

\emph{Case 2}: $p_o=p_v =1$, and $X_e$ non constant.  Then one has $a_1(n) = -a_2(n) = n$, so the path $\gamma^r_0(n)$ alternates one step to the right and one step downward. Hence $S(\gamma^r_0(n))$ is the sum of $n$ symmetric random variables, each distributed as $X_1-X_2$, with $X_1$ and $X_2$ independently drawn from $\mathcal {L}(X_e)$. Observe that this law cannot degenerate to the Dirac delta in $0$. Applying Lemma \ref{CLT} and the symmetry of the law of $S(\gamma^r_0(n))$, the result \eqref{limiti} is obtained also in this case.

\medskip

Finally, since
$ \lim_{K_2 \to +\infty } \mathbb{P} (F^r_{3}  (\bar m, \bar n , K_2) )     =1$, one can choose $\bar K_2$ in such a way that
\begin{equation}\label{numerodue3}
\mathbb{P} ( F_3^r (  \bar m,\bar n, \bar K_2) ) \geq 1 -\frac{\varepsilon}{4}.
\end{equation}
Putting together the inequalities \eqref{numerouno},  \eqref{numerodue2}, \eqref{numerodue3}, one arrives to the desired inequality \eqref{probF}.
\end{proof}

Taking into account the relation \eqref{contributo} and the definition \eqref{contributo2} we get the following statement
\begin{equation}\label{bounds}
  G^r   (  \bar  m,   \bar n, \bar  K_1,  \bar K_2) \text{ holds }
\Rightarrow
  S(\eta^r_+( \bar n))\in [\bar K_1, \bar K_3],\,\,\,\, S(\eta^r_- ( \bar n))  \in  [-\bar K_3, -\bar K_1],
\end{equation}
where $\bar K_3=\bar K_1+\bar K_2 \ell(\bar n)$. 
As a consequence, when concatening a given path with the minimum and the maximum path on a good block we can always keep 
the sum under control.

\medskip
\section{Proof of the main results}
The proof of Theorem \ref{primo} proceeds along the following two steps: first construct a binary random field on the blocks such that the good ones percolate from the origin with positive probability; then choose adaptively and concatenate paths within each block of a percolating sequence, keeping the partial sums under control.

\emph{Proof of Theorem \ref{primo}}. The first step is to translate all the events and random variables defined so far, computing them on each block $B^u_{a_1(n),a_2(n)}(\mathbf {b})=:B^u(\mathbf {b})$, for $u \in \{r,l\}$ and $\mathbf {b}\in \Z^2$. Next define
$$
(\mathbf {X}_{B^u(\mathbf {b})}, \mathbf {Z}_{B^u (\mathbf {b})})=\{X_e, Z_e, e \in B^u(\mathbf {b}) \}    ,
$$
 for any $\mathbf {b} \in \Z^2$. For a random variable of the form $\xi(O)=g(\mathbf {X}_{B^u (O)}, \mathbf {Z}_{B^u (O)})$
define the translated random variable
$$
\xi(\mathbf {b})=g(\mathbf {X}_{B^u (\mathbf {b})}, \mathbf {Z}_{B^u (\mathbf {b})}),
$$
for any $\mathbf {b}=(b_x,b_y) \in \Z^2$. For translations of an event we use a similar notation.  We also provide
independent copies of the vector of symmetrized random variables $T_{\sigma}(\gamma_i^u(\bar n))$, $i=1,\ldots,m$, with $u \in \{r,l\}$ which are assigned to the translated paths
$$
\gamma_i^u(\mathbf {b})=\gamma_i^u(\bar n)+b_xA_2(\bar n)+b_yA_3 (\bar n)
$$
inside each block $B^u(\mathbf{b})$, which will be called $T_{\sigma}(\gamma_i^u(\mathbf {b}))$, for $i=1,\ldots,m$.

\medskip

At this point we define
\begin{equation}\label{fieldBer}
  J^u (\mathbf {b}) = 1_{ G^u (\mathbf {b}) } , \text{ for } \mathbf {b} \in \Z^2, \text{ and } u \in  \{r,l\},
\end{equation}
where $G^u (\mathbf {b})=G^u (  \bar  m,   \bar n, \bar  K_1,  \bar K_2)(\mathbf {b})$.
 When $G^u (\mathbf {b})$ is realized we say that the block ${B}^u (\mathbf{b})$ is good. This is a random field on the edges  $\vec {E}_2$     of the oriented square lattice $ \vec G_2  $, with $J^r (\mathbf {b})$ assigned to the oriented edge from $\mathbf {b} =(b_x,b_y)$ to  $(b_x+1,b_y)$ and $J^l     ( \mathbf {b} ) $ assigned to the oriented edge from $(b_x,b_y)$ to  $(b_x,b_y+1)$. Notice that for any $\mathbf {b}=(b_x, b_y) \in \Z^2$, the pairs of blocks
  $({B}^r (\mathbf{b}),{B}^l (\mathbf{b}))$ and $({B}^r (b_x,b_y), {B}^l (b_x+1,b_y-1))$ share some vertical edges, and the corresponding random variables enter in the definition of the goodness of a block. As a result each of the pairs $(J^r( \mathbf {b} ), J^l ( \mathbf {b} ))$ and $(J^r(b_x,b_y), J^l (b_x+1,b_y-1 ))$ is not independent.

  Nonetheless, for any $h \in \Z$, the field
  $$
 {\mathbf{J}}^{[h]}= ( J^u   (  \mathbf {b} ) : \mathbf {b}=(b_x,b_y)     \in \Z^2, b_x +b_y = h ,\,\,\,  u \in \{ l,r \}   )
  $$
  is $1$-dependent and it is invariant under the (right) translation, defined as
  $$
  ((b_x,b_y),l) \mapsto ((b_x,b_y),r),\,\,\, ((b_x,b_y),r) \mapsto ((b_x+1,b_y-1),r).
  $$

 To verify the $1$-dependence property take $\mathbf {b}_i=(b_{x,i},b_{y,i})\in \Z^2$, with $b_{x,i}+b_{y,i} = h$, and $u_i \in \{r,l\}$, for $i=1,\ldots,l$. Suppose that $b_{x,i+1}-b_{x,i}\geq 1$,  and in case $b_{x,i+1}-b_{x,i}= 1$ it is forbidden that both $u_i=r$ and $u_{i+1}=l$ hold, for $i=1,\ldots,l-1$. This guarantees that the blocks $B^{u_i}(\mathbf {b}_i)$, for $i=1,\ldots,l$ are disjoint, equivalently that the parallelograms $R^{u_i}(\mathbf {b}_i)$ are not adjacent, for $i=1,\ldots,l$: then the random variables
  \begin{equation}\label{onedep}
  J^{u_i}   (  \mathbf {b}_i ), \,\,\,\, i=1,\ldots,l
  \end{equation}
 are mutually independent.

Next we use Theorem (7.65) in \cite{Grimmett}  with $d=k=1$ (for the original result see \cite{LSS}). It ensures that, for any $p \in (0,1)$, there exists $\varepsilon >0  $
  such that, when $ \mathbb{P} (J^u  (   \mathbf {b}   ) =1 ) \geq 1- \varepsilon$ holds,  a Bernoulli field 
  $$
  {\mathbf{W}}^{[h]}= ( W^u   (  \mathbf {b} ) : \mathbf {b}=(b_x,b_y)     \in \Z^2, b_x +b_y = h ,\,\,\,  u \in \{ l,r \}   )
  $$
  with parameter $p $ can be constructed, such that
  \begin{equation}\label{domina}
     J^u (   \mathbf {b}   ) \geq W^u ( \mathbf {b}  )      \text{ for } u \in \{ r,l \} \text{ and }   \mathbf {b}=(b_x,b_y)  \in \Z^2, \text{ with } b_x+b_y=h.
  \end{equation}

   On the other hand the collection of one-dimensional fields $( {\mathbf{J}}^{[h]}   : h \in \Z    ) $ is i.i.d., therefore one can take the fields $( {\mathbf{W}}^{[h]}   : h \in \Z    ) $ i.i.d. as well.
Now, consider  $ ( W^u ( \mathbf {b} ): u \in \{l,r\} ,  \mathbf {b}     \in \Z^2  )$ as  a field on the edges $\vec E_2$ of the oriented square lattice $ \vec G_2  $. As before, the random variable $W^r(b_x,b_y)$ is placed on the edge from $(b_x,b_y)$ to $(b_{x}+1,b_{y})$ and the random variable $W^l(b_x,b_y)$ is placed on the edge from $(b_x,b_y)$ to $(b_{x},b_{y}+1)$.
The dominance relation \eqref{domina} guarantees that when the field $ ( W^u ( \mathbf {b} ): u \in \{l,r\} ,  \mathbf {b}     \in \Z^2  )$ percolates from the origin, the same is true for the field $ ( J^u ( \mathbf {b} ): u \in \{l,r\} ,  \mathbf {b}     \in \Z^2  )$. This means that if there exists a sequence $((\mathbf {b}_k, u_k), k \in \mathbb {N})$,   with  $W^{u_k} ( \mathbf {b}_k )=1$ and
\begin{equation}\label{uenne}
(b_{k+1,x}, b_{k+1,y})=(b_{x,k}+\delta_{u_k,r},b_{y,k}+\delta_{u_k,l}),
\end{equation}
where $\delta$ is the Kronecker delta, all the corresponding blocks $B^{u_k}(\mathbf {b}_k)$ will be good, for all $k \in \N$. Now it is well known  that there exists a critical threshold $p_c(\vec G_2) \in (0,1)$ such that a Bernoulli field with $p>p_c(\vec G_2)$ percolates from the origin with positive probability (see  \cite{Durrett} and \cite{BBS94}). As a consequence, provided $\varepsilon$ appearing in \eqref{probF} of Lemma \ref{lemmaT} is small enough,  the field $ ( J^u ( \mathbf {b} ): u \in \{l,r\} ,  \mathbf {b}     \in \Z^2  )$ percolates as well. This means that there exists, with positive probability, an infinite self-avoiding path of good blocks, starting from the origin. From now on we suppose that $\bar m, \bar n, \bar K_1$ and $\bar K_2$ have been chosen in such a way that the value of $\varepsilon$ appearing in \eqref{probF} is so small to guarantee a positive probability of percolation from the origin.

The second step of the proof consists in defining an infinite self-avoiding path $\eta^{*}$, starting from the origin,  constructed from a percolating path $((\mathbf {b}_k,u_k), k \in \mathbb {N})$ of good blocks $B^{u_k}(\mathbf {b}_k)$, with $\mathbf {b}_0=O$ (see Figure 3). 

\begin{figure}[ht]
\centering
\includegraphics[width=90mm]{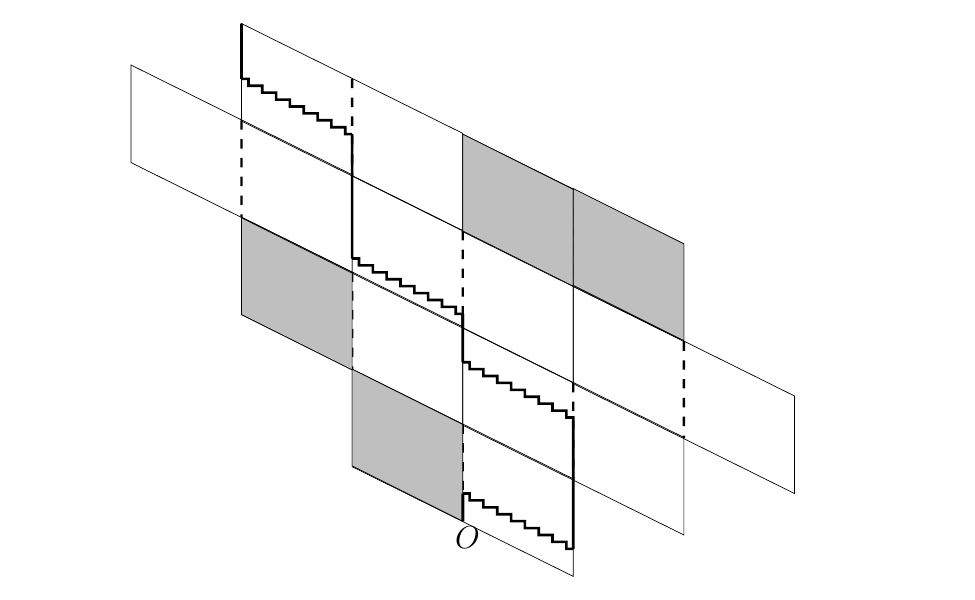}
   \caption{\small{The construction of the path $\eta^*$. White blocks are good.
   }}
\end{figure}

It will be proved that $sup_h|S_h(\eta^{*})|\leq C$, $C$ being a suitable positive constant. The path $\eta^{*}$ is constructed by successive concatenations of the minimum and maximum paths
$$
\eta_-^{u_k}(\mathbf {b}_k)=\eta_{i_k}^{u_k}(\mathbf {b}_k), \,\,\,\, \eta_+^{u_k}(\mathbf {b}_k)=\eta_{j_k}^{u_k}(\mathbf {b}_k),
$$
for the block $B^{u_k}(\mathbf {b}_k)$, where the indices $i_k=i_{u_k}(\mathbf {b}_k)$ and $j_k=j_{u_k}(\mathbf {b}_k)$ are defined analogously to \eqref{indicer} and \eqref{indicer2}. For $\mathbf {b}_k=(b_{x,k},b_{y,k})$, these paths run from $b_{x,k}A_2(\bar n)+b_{y,k}A_3(\bar n)$ to
 $b_{x,k+1}A_2(\bar n)+b_{y,k+1}A_3(\bar n)$.

Recall that by \eqref{bounds} one has
\begin{equation}\label{pathtrans}
  0<\bar K_1\leq S(\eta^{u_k}_+(\mathbf {b}_k))\leq \bar K_3,\,\,\,\, -\bar K_3 \leq S(\eta^{u_k}_-(\mathbf {b}_k)) \leq -\bar K_1<0.
\end{equation}

For $k \in \mathbb {N}$, we denote by $\eta^*_k$ the path starting from the origin $O$ and ending in the site $b_{x,k+1}A_2(\bar n)+b_{y,k+1}A_3(\bar n)$, constructed by the following recursion.
Now suppose that $\eta^*_k$ has been defined, set $s_k=S(\eta^*_k)$ and define
\begin{equation}\label{recurs}
\eta^*_{k+1}=
           \eta^*_k \odot \eta_{-\text{sign} (s_k)}^{u_{k}}(\mathbf {b}_{k}),\,\,\, k \in \N
\end{equation} 
where for definiteness the sign of $0$ is taken to be $-1$.  Setting $\eta^*_0 = \emptyset$, this holds also for $k=0$. As a consequence
\begin{equation}\label{quasimarkov}
s_{k+1}=s_k+S(\eta_{-\text{sign} (s_k)}^{u_{k}}(\mathbf {b}_{k})), \,\,\, k \in \N, s_0=0.
\end{equation}
In other words the last term in the 
concatenation tries to reverse the sign of the current sum on the path $\eta^*$.

For later use, notice that the sequence $(s_k, k \in \N)$ is an homogeneous Markov process: indeed the law of the increment $S(\eta_{-\text{sign} (s_k)}^{u_{k}}(\mathbf {b}_{k}))$ depends only on $\text{sign} (s_k)$ but not on $u_{k}$ and $\mathbf {b}_{k}$, as it results from \eqref{orsa}. Taking into account \eqref{pathtrans}, and the opposite signs of the two summands at the r.h.s. of \eqref{quasimarkov}, we have  $|s_k| \leq \bar K_3$, for any $k \in \mathbb {N}$.

It remains to bound the sum $S_n(\eta^{*})$, when $|\eta^{*}_{k}|<n<|\eta^{*}_{k+1}|$, for some $k \in \N$. Since
$G^{u_{k}}(\bar {m}, \bar {n}  , \bar K_1 , \bar K_2        )   {(\mathbf{b}_{k})}        $ is realized, for any $k \in \N$, the sum of the absolute values of the $|X_e|$'s along each of the paths $\eta_{-\text{sign} (s_k)}^{u_{k+1}}(\mathbf {b}_{k+1})$ is bounded by $\bar K_3$.
 As a consequence $|S_n(\eta^{*})|\leq 2 \bar K_3$ for any $n \in \mathbb {N}$. This ends the proof of Theorem \ref{primo}.
$ \Box$

\medskip

Moving towards the proof of Theorem \ref{Markovtorna} we analyze the behaviour of the Markov process defined in \eqref{quasimarkov}. Not surprisingly, it is related to the nature of the support $\mathcal {X}$ of the random variable $X_e$.

Fix any $0 \neq \bar y \in \mathcal {X}$ and consider the rescaled set $\bar y^{-1}\mathcal {X}$. If this is a finite subset of $\mathbb {Q}$ we say that $\mathcal {X}$ is \emph {finite rational}: indeed in this case there exists $\rho >0$ such that  $\mathcal {X} \subset \rho \Z$. If $\bar y^{-1}\mathcal {X}$ is a countable subset of $\mathbb {Q}$ we say that $\mathcal {X}$ is \emph {countably rational}. Finally, we say that $\mathcal {X}$ is \emph {irrational} if there exist $0 \neq \bar y_i \in \mathcal {X}$, $  i=1,2 $ with $\frac {\bar y_2}{\bar y_1} \notin \mathbb{Q}$.

The first and the third case are taken care by the following Lemma. Later on, we will reduce the second case to the first one.

Notice that when $|x|\leq \bar C, \forall x \in \mathcal {X}$ (a fortiori when $\mathcal {X}$ is finite) then one can choose $\bar K_1=(3m+1)\bar C$ and $\bar K_2=\bar C$ in the definition of a good block so the events $F_1^u(m,n,K_1)$ and $F_3^u(m,n,K_2)$ defined in \eqref{F1} and \eqref{F3} are equal to the whole sample space.

On the other hand it should be observed that when $\mathcal {X}$ is unbounded  only part of it enters in the transition kernel of \eqref{quasimarkov}. Thus, in the irrational case, we always suppose to have chosen $\bar K_2$ so large that both $\bar y_1$ and $\bar y_2$ appearing in the above definition are in $[-\bar K_2, \bar K_2]$.

\begin{lemma} \label{camminimarkov}
Let $p_0, p_v <1$. Consider two independent i.i.d. sequences of random variables $\zeta^-_k \Legge -S(\eta_{-}^{u}(\mathbf {b}))$ and $\zeta^+_k \Legge S(\eta_{+}^{u}(\mathbf {b}))$, both laws conditional to the goodness of the block $B^{u}(\mathbf {b})$. The Markov process
\begin{equation}\label{Schelling}
s_{k+1}^x=s_k^x-\zeta^{-}_{k}1_{\{s_k>0\}}+\zeta^{+}_{k}1_{\{s_k\leq 0\}}, \,\,\, k \in \N, s_0^x=x
\end{equation}
has the following property:
\begin{itemize}
\item[i.] if $\mathcal {X}$ is finite rational, the process $(s_k^x, k \in \N)$ visits the origin infinitely often, for any
    \begin{equation}\label{parigi}
    x=\sum_{i=1}^{2N}x_i,\,\,\,\, x_i \in \pm \mathcal{X} :=-\mathcal{X} \cup \mathcal{X} ;
    \end{equation}
\item[ii.] if $\mathcal {X}$ is irrational, the process $(s^x_k , k \in \N   )$, for any  $x \in \mathbb{R}$, visits any neighborhood of the origin infinitely often.
\end{itemize}

\end{lemma}

The proof of i. and ii. are of different nature, the first being inherently algebraic, whereas the second uses arguments from dynamical systems.

\begin{proof}[Proof of i.]

Since $p_o$ and $p_v$ are smaller than $1$ we can change the sign of the sum of a  path simply by changing all the signs of the $Y_e$'s associated to its edges; moreover if all the signs of the $Y_e$'s within a block are changed then a good block remains good. As a consequence the support 
$\mathcal{S}$ of the (positive) increments $-S(\eta_{-}^{u}(\mathbf {b}))$ and $S(\eta_{+}^{u}(\mathbf {b}))$, conditional to the goodness of $B^{u}(\mathbf {b})$ are equal and contained in the additive group $\text{gr}(\mathcal {X})$ generated by $\mathcal {X}$.  By \eqref{bounds}, one has the inclusion $\mathcal {S}\subset [\bar {K}_1, \bar {K}_3]$.

 Let $R(x)$ be the set of states reachable from $x$ in a finite number of steps of the chain. It is clear that $R(x) \subset x+\text{gr}(\mathcal {S})$; in particular if $y \in R(0)$, then $y \in \text{gr}(\mathcal {S})$, that is
\begin{equation}\label{reachable1}
y=\xi_1+\ldots+\xi_n-\xi_{n+1}-\ldots-\xi_{n+m},\,\,\,\, \xi_i \in \mathcal {S},\,\,\,\, i=1,\ldots,n+m.
\end{equation}
From this
\begin{equation}\label{reachable2}
y-\xi_1-\ldots-\xi_n+\xi_{n+1}+\ldots+\xi_{n+m}=0,\,\,\,\, \xi_i \in \mathcal {S},\,\,\,\, i=1,\ldots,n+m.
\end{equation}
Now we can achieve the total sum at the l.h.s. of \eqref{reachable2} by choosing
 the order in which each of the summands  enter in the sum in the following way. Starting from $y$, choose to add or subtract one of the terms $\xi_i$, $i=1,\ldots,n+m$, according to the rule:

\begin{itemize}
  \item add one of the $\xi_i, i=n+1,\ldots,m$ when the current sum is negative or zero;
  \item subtract one of the $\xi_i, i=1,\ldots,n$ when the current sum is positive.
\end{itemize}
After $n+m$ steps the result is $0$. But this is exactly a trajectory, with positive probability, of $n+m$ steps of the process \eqref{Schelling}, thus $0 \in R(y)$. This proves that $R(0)$ is an irreducible class which is contained in a finite subset of
$[- \bar K_3,  \bar K_3]$  (see the end of the proof of Theorem \ref{primo}). So $0$ is recurrent.

 The same argument proves the recurrence of $0$ for any starting point $x \in \text{gr}(\mathcal{S})$. It remains to prove that any $x$ of the form \eqref{parigi} belongs to $\text{gr}(\mathcal {S})$. To this purpose observe that one can always suppose that $N=(3m+1)D$ for some integer $D$; if $N$ is not divisible by $3m+1$ add and subtract a fixed non zero element of $\mathcal {X}$. Now one can partition the $2N$ indices in the sum \eqref{parigi} in $2D$ sets, call them $W_j, j=1,\ldots,2D$, of cardinality $3m+1$. Let us choose a value $\sigma$ in the support of $S(\gamma_+^u(\bar n))$. It is easily verified that
$$
\xi_j=\sum_{i \in W_j} x_i + (-1)^j\sigma,  \text{ }j=1,\ldots,2D
$$
belongs to $\mathcal {S}$ (the first sum corresponding to the contribution of the vertical boundaries and the second to that of  the path $S(\gamma_{\pm}^{u}(\mathbf {b}))$ inside a block $B^u(\mathbf {b})$) and
$$
x=\sum_{i=1}^{2N}x_i=\sum_{j=1}^{2D} \xi_j \in \text{gr}(\mathcal S).
$$
\end{proof}
\begin{proof}[Proof of ii.]
In this case, we aim to prove that the number of times the process enters in any neighbourhood of $0$ is a.s. infinite. One can always assume that the initial point $x$ belongs to the invariant interval $[-\bar K_3, \bar K_3]$, since this can be reached in a finite number of steps. Then, for any positive integer $N$ divide $[-\bar K_3, \bar K_3]$ in subintervals of the form
$I_h=(\bar K_3-h\varepsilon, \bar K_3-(h-1)\varepsilon]$, for $h=1,\ldots,2N$, where $\varepsilon=\frac {{\bar K}_3}{N}$ (the point $-\bar K_3$ is added to the last interval). We will prove that there exists $\delta>0$ and $m_h \in \mathbb {N}$  such that
\begin{equation}\label{uniform}
\mathbb {P}(|s_{m_h}^{\zeta}| \leq 2\varepsilon) \geq \delta, \forall \zeta \in \bar I_h, \text{ for } h=1,\ldots, N .
\end{equation}
Notice that we can choose $m_h =0$, for $h=N-1$ and $h=N$. Suppose now that \eqref{uniform} holds.  Starting from $T_0^h=0$, for each of the intervals $I_h, h=1,\ldots,N-2$   define the sequence of successive return times
$$
T_{k+1}^h=\inf \{n>T_{k}^h+m_h: s_n \in I_h\},\,\,\, k=0,1,\ldots
$$
(where $\inf \emptyset =+\infty$). Consider the events
\begin{equation}\label{hitting}
E_k^{h}=\{|s^{x}_{T_{k-1}^h+m_h}|\leq 2\epsilon\},
\end{equation}
for $k=1,\ldots$, and for each $h =1, \ldots , N$, the filtration
$
 \left (\mathcal {F}^h_k=         \sigma \{s_n^{x}, n \leq T_k^h\},  \text{ } k \in \N    \right ).
$

It is immediately verified that $E_k^{h} \in \mathcal {F}^h_k$. Moreover, from \eqref{uniform} one has
$$
\mathbb P(E_k^{h}|\mathcal {F}^h_{k-1})\geq \delta, \text{   if    } T_{k-1}^h < +\infty,
$$
for $k \geq 2$. From L\'evy's extension of the Borel-Cantelli Lemmas  (see \cite{Williams}, p. 124) it is obtained that
\begin{equation}\label{Will}
   T_k^h <+ \infty  \text{ for } k \in \mathbb{N}   \Rightarrow        Z^h  :   =\sum_{k=1}^{\infty} 1_{E_k^h}=+\infty.
\end{equation}
Now let
$$
H=\inf \{h=1,\ldots,N: T_k^h<+\infty , \text{ for } k\in \N \}
$$
and notice that $H$ is a.s. finite since the number of visits to the positive axis is a.s. infinite. Now
$$
\mathbb {P}(|   s_n^{x}|\leq 2\epsilon  , \text{ i.o.}      )\geq \sum_{h=1}^{N-2} \mathbb {P} (H=h, Z^h=    +\infty)+\mathbb {P}(H=N-1) +\mathbb {P}(H=N)
$$
$$
\geq \sum_{h=1}^{N-2} \mathbb {P} (H=h,  T^h_k < +\infty \text { for } k \in \mathbb{N},
 Z^h=    +\infty)+\mathbb {P}(H=N-1) +\mathbb {P}(H=N)
$$
$$
= \sum_{h=1}^{N-2} \mathbb {P} (H=h,  T^h_k < + \infty \text { for } k \in \mathbb{N}  )+\mathbb {P}(H=N-1) +\mathbb {P}(H=N)
$$
$$
= \sum_{h=1}^{N-2} \mathbb {P} (H=h  )+\mathbb {P}(H=N-1) +\mathbb {P}(H=N) =1,
$$
where we have used \eqref{Will} to get the first equality. By consequence it remains only to prove \eqref{uniform}.

 Now recall that in the irrational case there exists $\bar x, \bar y \in \pm \mathcal {X}$ both positive with $0<\theta=\frac {\bar x}{\bar y} <1$ irrational. Suppose first that $\bar x$ and $\bar y$ are both atoms. Then, since $t=\ell(\bar n)+3m+1$ is the length of the paths constructed inside each block, we have $t\bar x$ and $t\bar y$ both belong to $\mathcal {S}$ (this may require larger values of the parameters for a good block,  which is always possible to specify in advance). Redefining the values of $\bar x$ and $\bar y$  we can set $t=1$ in what follows. Next consider the dynamical system
\begin{equation}\label{dinam}
w_{n+1}=w_n-\bar x 1_{\{w_n>0\}}+ \bar y 1_{\{w_n\leq 0\}}, \text{ for } n \in \mathbb{N} .
\end{equation}
This dynamical system represents  a  transition, having positive probability, of the Markov process of interest when the contribution of negative paths is $-\bar x$ and that of positive paths is $\bar y$. This dynamical system started from any point  enters in the
invariant interval $[-\bar y, \bar y]$ after a certain finite number of steps. Next we rescale the system dividing by $\bar y$, getting for the rescaled sequence $\tilde w_n=w_n/\bar y\ \in [-1,1]$ the recursion
\begin{equation}\label{dinam2}
\tilde w_{n+1}=\tilde w_n-\theta 1_{\{\tilde w_n>0\}}+1_{\{\tilde w_n\leq 0\}}=:f(\tilde w_n).
\end{equation}
Now let $\tilde w_n^{z}$ be the iterates of \eqref{dinam2} started from $\tilde w_0^{z}=z$, $0<z\leq 1$. Fix $\epsilon>0$ and suppose that $m (z)$ is the smallest integer such that $-\epsilon \leq \tilde w_{m(z)}^{z} \leq 0$. Then it can be easily verified that for any $\tilde z \in (z,z+\epsilon]$ it remains $|\tilde w_{m(z)}^{\tilde z}|\leq \epsilon$. Since the length of the intervals $I_h$ is precisely $\epsilon$, if we show that such an $m$ exists, this will end the proof for the irrational atomic case. To this purpose notice that the induced map on the interval $(0,1]$
$$
\tilde f (w)=f(w)1_{\{f(w)>0\}}+(f\circ f) (w)1_{\{f(w)\leq 0\}},\,\,\, w\in (0,1]
$$
has the form
$$
\tilde f (w)=\left \{ \begin{array}{lcl}
1-\theta + w  & \text{ if } & w \in (0, \theta] \\
w- \theta & \text{ if } & w \in ( \theta, 1]
\end{array}
\right .
$$
which coincides with the rotation map on the circle (parameterized by $(0,1]$) with an (irrational) angle $\theta $. Since it is well known that all the orbits of an irrational rotation map are dense, this show that $m (z)$ exists (see e.g. \cite{KH95} p. 27).

Next we turn to the general case in which $\bar x, \bar y$ belong to $\pm \mathcal {X}$, but they are not necessarily atoms.
This requires to control a set of ``perturbed" trajectories, to which a suitable continuity argument has to be applied.
So, let us to consider the image of the functions $(w_n^{\tilde {z}}(u_0,\ldots,u_{n-1}), n=1,\ldots,m(z))$ defined below, for $\tilde {z} \in [z,z+\epsilon]$, with $0<z\leq 1$, and $|u_i|<\varrho $, for $i=1,\ldots,m(z)-1$, where $ \varrho>0$ is suitably small. These functions are defined by the recursion
$$
w_{n+1}^{\tilde z}(u_0,\ldots,u_{n-1},u_n)=w_n^{\tilde z}(u_0,\ldots,u_{n-1})+
$$
\begin{equation}\label{dinamnatom}
+(\bar y+u_{n}) 1_{\{w_n^{\tilde z}(u_0,\ldots,u_{n-1})\leq 0\}}-(\bar x+u_{n}) 1_{\{w_n^{\tilde z}(u_0,\ldots,u_{n-1})>0\}} ,\,\,\, w_0={\tilde z},
\end{equation}
 taken for $n=1,\ldots,m(z)-1$. When $\varrho $ is small enough, for $n=1,\ldots,m(z)-1$ we can guarantee that all the trajectories of \eqref{dinamnatom}, for any $|u_i|< \varrho$, $i=1,\ldots,m(z)-1$ are as close as desired to that of $w_n^{\tilde z}(0,\ldots,0)$, and in particular, for any $n=1,\ldots,m(z)-1$, they all lie either on the negative or on the positive side of the axis. As a consequence
$$
w_n^{z}(u_0,\ldots,u_{n-1})-w_n^{z+\epsilon}(u_0,\ldots,u_{n-1})=\epsilon, n=1,\ldots,m-1,m(z),
$$
which guarantees that, since $-\epsilon<w_{m(z)}^{z}(0,\ldots,0)\leq 0$ we have that the image of
$$
(\tilde {z},u_0,\ldots,u_{m(z)-1}) \in [z,z+\epsilon]\times (- \varrho  , \varrho)^{m(z)} \to w_{m(z)}^{\tilde z}(u_1,\ldots,u_{m(z)-1})
$$
is contained in the interval $(-\epsilon-\epsilon^*,\epsilon+\epsilon^*)$ for $\varrho$ suitably small, for any possible choice of $\epsilon^*>0$. In view of the assumption that the open balls of radius $\varrho $ around both $\bar x$ and $\bar y$ are charged with positive probability by both the laws of $-S(\eta_{-}^{u_{k}}(\mathbf {b}^{k}))$ and $S(\eta_{+}^{u_{k}}(\mathbf {b}^{k}))$, respectively, this ends the proof.
\end{proof}

\medskip
\begin{rem}
The Markov process \eqref{Schelling} has a peculiar form. Indeed notice that
if we replace $S(\eta^u_{\pm}(\mathbf {b}))$ with $S(\gamma^u_{\pm}(\mathbf {b}))$ defined in \eqref
{contributo?}, that is we neglect the contribution to the sum coming from the vertical boundaries of each block, one would get increments with the symmetry property $-S(\gamma^u_{-}(\mathbf {b})) \Legge S(\gamma^u_{+}(\mathbf {b}))$.
 In this case $(|s^x_k|, k \in \mathbb {N})$ is again a Markov process, of the type known in the literature as the von Schelling process
 \cite{PW2011}, or with a different name, the absolute value chain \cite{Knight}.
\end{rem}

\begin{proof}[Proof of Theorem \ref{Markovtorna}] First of all, using ergodicity w.r.t. vertical translations, with the choice of the block parameters made in Theorem \ref{primo}, one finds a.s. an oriented path of percolating blocks, starting from $B^u(h,h)$ for some positive integer $h$ and $u \in \{l,r\}$. The path $\hat \eta$  is constructed by the concatenation of the vertical path $\beta^0_{0,2(3m+1)h}$ joining the origin with the vertex $V=(0,2(3m+1)h)$ with edges placed on the $y$-axis, and an infinite path $\eta^* $ constructed according to the rules \eqref{recurs} and \eqref{quasimarkov}, but starting from $B^u(h,h)$ (see Figure 4). It is clear that the contribution of the vertical part of the path gives an initial value for the recursion \eqref{quasimarkov} which is in general different from $0$. If $\mathcal {X}$ is either finite rational or irrational, Lemma \ref{camminimarkov} directly allows to prove the theorem. For the first case notice indeed  that the initial value for \eqref{quasimarkov} has always to be a sum of an even number of elements of $\pm \mathcal {X}$.

\begin{figure}[ht]
\centering
\includegraphics[width=90mm]{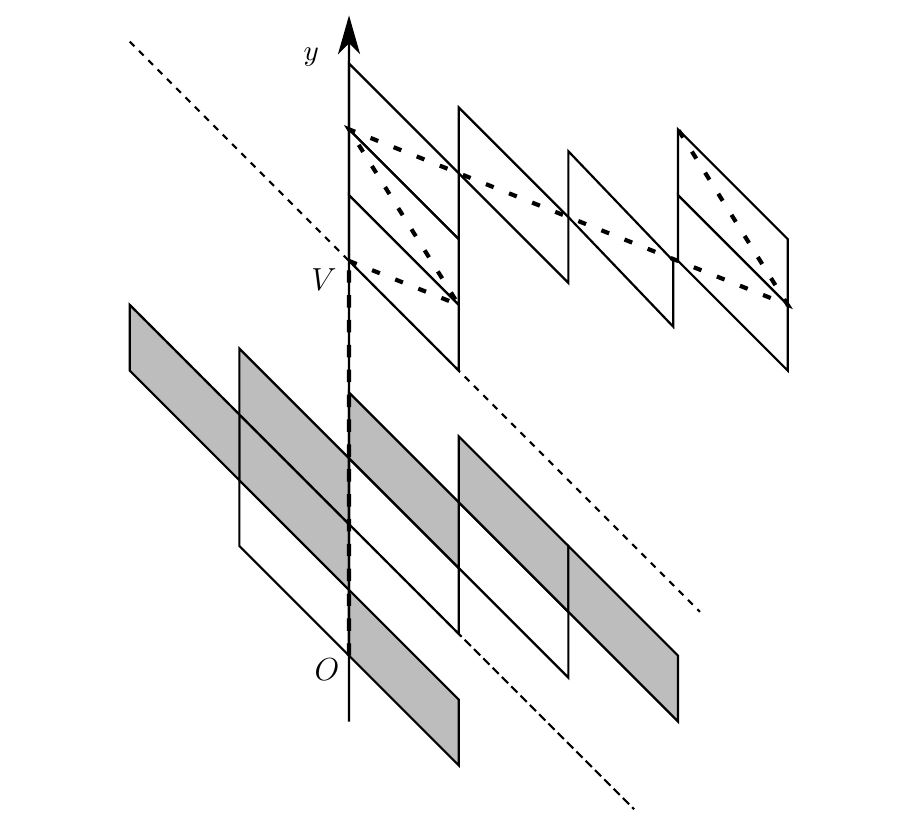}
   \caption{\small{The construction of the path $\hat \eta$ in Theorem \ref{Markovtorna}. The path $\hat \eta $ is
the concatenation of a vertical path from $O$ to $V$ with the infinite path $\eta^*$ started in $V$, represented with dashed segments in the figure.
   }}
\end{figure}

Therefore, in the remaining part of the proof, we have to take care only of the countably rational case.
We will reduce this case to
the finite rational one; indeed we will prove the existence of a self-avoiding
path $\hat \eta$ with bounded partial sums  starting from the origin that not only visits infinitely
often any neighborhood of $0$, but visits infinitely often the origin itself.

In the definition of good blocks, even constraining the $X_e$'s within a block to lie in some finite subset of $\mathcal{X}$, it is possible to keep the probability that a block is good arbitrarily close to $1$ and thus the probability of a percolating path from the origin to be positive. However, since the a.s. existence of a percolating path is guaranteed only by shifting the starting point vertically, one cannot be sure that the values appearing on the vertical path joining its starting point with the origin lie inside the allowed subset.
This requires a ``revised" definition of good block where the support  $\mathcal{X} $ is replaced by a sequence of finite subsets whose size is \emph{adaptively} adjusted.

So let $\mathcal{X}^*_0 \subset \mathcal{X}^{**}_0$ be two \emph{finite} subsets of $\mathcal{X}$ and define
\begin{equation}\label{X1}
  \bar F_1^u(m,n_0, \mathcal{X}^*_0)=\left \{  X_e \in \mathcal{X}^*_0: e \in E^{u,r}(m,n_0) \cup  E^{u,l}(m,n_0)  \right \}       
\end{equation}
\begin{equation*}\label{X2}
 \bar F_2^u (m,n_0, \mathcal{X}^*_0) =F_2^u (m,n_0, (3m+1)\sup |\mathcal{X}^*_0|) =
\end{equation*}
\begin{equation}
\label{X2}
=\left \{  S(\gamma_{+}^u(n_0)) \geq 2(3m+1)\sup |\mathcal{X}^*_0| \right \}
 \cap \left \{  S(\gamma_{-}^u(n_0)) \leq  -2(3m+1)\sup |\mathcal{X}^*_0|  \right \}  ,
\end{equation}
\begin{equation}\label{X3}
\bar F_3^u(m,n,\mathcal{X}_0^{**}) =\{ X_e  \in \mathcal{X}^{**}_0 : e \in \bigcup_{j =0}^{m-1} \gamma_{j}^u(n_0)  \} .
\end{equation}
A ``revised" good block is a block for which the event
\begin{equation}\label{G8}
\bar G^u   (m,n_0, \mathcal{X}^*_0, \mathcal{X}^{**}_0)     =      F_0^u(m,n_0) \cap     \bar F_1^u(m,n_0, \mathcal{X}^*_0) \cap  
\bar F_2^u (m,n_0,  \mathcal{X}^*_0) \cap     \bar F_3^u(m,n_0,\mathcal{X}^{**}_0)
\end{equation}
is realized. It is assumed that the parameters of the block guarantee that the probability of a good block is larger than $1-\varepsilon$, where $\varepsilon$ is fixed once and for all and it  is so small to ensure a positive probability of percolation of good blocks from the origin. More precisely, as in the proof of Lemma \ref{lemmaT}, we take $m= \lfloor \log_2 \varepsilon^{-1} \rfloor +4$, $\mathcal{X}^*_0$, $n_0$ and $\mathcal{X}^{**}_0$ large enough to control that the probabilities of \eqref{X1}, \eqref{F0} and \eqref{X2}, and finally \eqref{X3}, are larger than $1-\frac {\varepsilon}{4}$.

Now define the \emph{cluster} $\mathcal{C}(O)$ as the collection of vertices
belonging to the oriented paths of revised good blocks starting from $B^r(O)$ or $B^l(O)$, and let $C_1$ be the event that this is collection is infinite.
If $C_1$ is not realized, define the random variable
$$
H_{1}=\inf \{t >0 : B^u( b_x, b_y) \cap \mathcal{ C}(O)  = \emptyset \text{ for } b_x+b_y = 2t -1, u \in \{r,l\}   \}.
$$
$H_1$ is a stopping time w.r.t. the filtration $(\mathcal {G}_h, h \in \mathbb {N})$, where $\mathcal {G}_h$ is the $\sigma$-algebra generated by all the variables associated to the blocks $B^u(b_x,b_y)$, with $0\leq b_x+b_y\leq 2h-1$, $u \in \{r,l\}$. 

Next we update the definition of good block choosing  
\begin{equation}\label{iso}
\mathcal{X}^*_1 = \mathcal{X}^*_0 \cup \{X_e: e \in \beta^0_{0,2(3m+1)h}\},
\end{equation}
and then select $n_1$ and $\mathcal{X}^{**}_1 \supset \mathcal{X}^*_1$ in such a way that the probability that $\bar G^u   (m,n_1, \mathcal{X}^*_1, \mathcal{X}^{**}_1)$ is realized is larger than $1-\varepsilon$. 

After this, construct the \emph{cluster} $\mathcal{C} (H_1,H_1)$  of vertices 
which belong to the oriented paths of good blocks starting from $B^r(H_1,H_1)$ or $B^l(H_1,H_1)$ and define the event $C_2$ that this cluster is infinite. We warn the reader that the oriented graph structure remains the same in spite of the fact that the size of blocks can change because $n_1$ has replaced $n_0$. Since all the variables associated to the blocks $B^u(b_x,b_y)$, with $b_x+b_y\geq 2H_1$, $u \in \{r,l\}$ are independent of the $\sigma$-algebra $\mathcal {G}_{H_1}$, it is
\begin{equation*}\label{condiriso0}
  \mP ( C_{2}  | \mathcal{G}_{H_1})       \geq  1-\varepsilon     .
\end{equation*} 

It should be clear that this argument can be iterated, so a sequence of stopping times $H_k$ and the corresponding events $C_k$, are defined for $k=1,\ldots,K$, where $K$ is the first index $k$ such that $C_k$ is realized (hence $H_k=+\infty$). If the event $C_k$ is realized, an infinite oriented path of revised good blocks exists, and if $C_k$ is not realized, $H_k$ indicates how many stripes of blocks one has to exclude before trying a new attempt for building the path, independently of the past ones. The construction can be always performed by keeping, for any $k \in \mathbb {N}$, 
\begin{equation*}\label{condiriso}
  \mP ( C_{k +1}  | \mathcal{G}_{H_k})       \geq  1-\varepsilon  .
\end{equation*}
on the event $\{H_k<\infty\}$. As a consequence, by the already cited Levy's extension of the Borel-Cantelli lemma, $K$ is finite a.s. See Figure 4, where $H_1=1 $, $H_2=3$, and $K=3$. 

Conditional to $\mathcal{G}_{H_K}$, consider the recursion \eqref{quasimarkov},
constructed over a percolating oriented path $(B^{u_p}(\mathbf {b}_p), p \in \N)$, with $\mathbf {b}_0=B^{u_0}(H_K,H_K)$, of $H_K$-adapted good blocks: it is still a Markov process, started from
\begin{equation}\label{condipasta}
s_0=\sum_{e \in \beta^0_{0,2(3m+1)H_K}} Y_e X_e. 
\end{equation} 
Using Lemma \ref{camminimarkov} one finally gets the recurrence of $0$.
\end{proof}

\medskip 
\begin{proof}[Proof of Proposition \ref{secondo}]
First of all we recall that for bond percolation the critical point for the square lattice
is $p_c =\frac{1}{2}$.  Therefore, if $\mathbb{P} (X_e =0 ) < \frac{1}{2}$
the edges where the random variables take the value zero do not percolate.
By continuity of the measure there exists $\delta>0 $ such that
$\mathbb{P} (|X_e| < \delta ) < \frac{1}{2} $, thus, with probability one, any path in $\Gamma_O$ has an edge $e $ such that $|X_e| \geq  \delta$. It is readily shown that this implies  
$   M_c (p_o,p_v ; \mathcal{L})    \geq \delta /2>0$, as done in  the following lemma

\begin{lemma}\label{facile}
Let  $(a_n \in \R : n \in \N)$ be a sequence of real  numbers, and $A\in (0,+\infty )$. If
$\sup_{n \in \N} |a_n|>A$, then $   \sup_{n \in \N}   |\sum_{k=1}^{n} a_k| >  \frac{A}{2} $.
\end{lemma}
\begin{proof}[Proof of Lemma]
By assumption there exists $\bar n \in \mathbb{N}$ such that
$|a_{\bar n }| > A$.  Let us fix such a $\bar n $.
 If $\bar n =1 $ then  $   \sup_{n \in \N}   |\sum_{k=1}^{n} a_k| \geq  | a_1 | >  A$.
If $\bar n >1 $ then
$$   \sup_{n \in \N}   |\sum_{k=1}^{n} a_k| \geq
  \sup \left \{   |\sum_{k=1}^{ \bar n -1} a_k|     , \,\,
 |\sum_{k=1}^{ \bar n } a_k|       \right \}   >  \frac{A}{2} .
$$
\end{proof}
\noindent \emph{Proof of Proposition \ref{secondo}, continued.}
In order to prove $b.$ it is enough to notice that if $\mathbb{P} (X_e =0 ) > \frac{1}{2}$ then there is bond percolation. Therefore
there exists with positive probability an infinite self-avoiding  path, starting from the origin,   using only edges $e $  with $X_e =0$.
\end{proof}

\begin{proof}[Proof of Theorem \ref{uniforme}]
In order to prove the  implication $\Rightarrow $ we prove that each of the following two conditions imply that $\bar{M}_c (p_o, p_v; \mathcal{L}) = +\infty $.

\begin{itemize}
\item[1.] The support of $X_e$ is unbounded.
\item[2.] The law of $X_e$ has a non zero atom.
    \end{itemize}

As far as item 1. is concerned, it is enough to notice, that, with probability $1$, there exists a vertex $u \in \Z^2$ with all $4$ incident edges carrying a value of $X_e$ which exceeds in absolute value any given constant $C$. This ends the proof for item 1.

Concerning item 2., suppose w.l.o.g. that $1$ is an atom. Then for any arbitrary large integer $L$ there exists a.s. a ball $\mathcal {B}(\mathbf{u},L)$ in the  $L_1$ norm, centered in some vertex $\mathbf {u}=(u_x,u_y) \in \Z^2$, with the following  property.
For any edge $e$ (seen as on open segment) inside the ball $\mathcal {B}(\mathbf {u},L)$, it is $X_e =1$, and moreover 
$$
e=\{(u_x+a,u_y+b),(u_x+a+1,u_y+b)\}\Rightarrow \text {sign}(Y_{e})=\text {sign} (a),
$$
$$
e=\{(u_x+a,u_y+b),(u_x+a,u_y+b+1)\}\Rightarrow \text {sign}(Y_{e})=\text {sign}(b),
$$
where  the sign of $0$ is taken to be  $+1$. 
In other words each oriented  edge inside the ball points always in the direction of the boundary. 
 Then it is not difficult to realize that any path from $\mathbf {u}$ to the boundary of the ball will have a sum equal to  $L$. This denies the possibility that $\bar{M}_c (p_o, p_v ; \mathcal{L})$ remains bounded.

Last  we prove the implication $ \Leftarrow $. 
This is shown by controlling the contribution of a path joining two sites on the same horizontal (or vertical) line. This is achieved by iterating 
suitable number of times a $4$--cycle, following in the direction which makes the current sum closer to zero.

First observe that,  by assumption,  the distribution has no atoms different from zero: as a consequence,  on any cycle $\sigma$ the sum $S(\sigma)$
 is either $0$, when all the edge variables $X_e$ on the cycle are $0$, otherwise it is different from zero a.s. In the latter case either $S(\sigma)$ or $S(-\sigma)=-S(\sigma)$ is positive (and the other negative). In particular this is true for the cycle $\sigma_p$ joining the vertices $(p,0)$, $(p+1,0)$, $(p+1,-1)$, $(p,-1)$ and $(p,0)$, with $p \in \mathbb {N}$, that we are going to use in the construction of the path. Obviously $|S(\sigma_p)|\leq 4\bar C$, where
  $$
\bar C  = \sup \{|x|  :x  \in \mathcal {X}\} .
$$
  Moreover if $S(\sigma_p)<0 $ then
 \begin{equation}\label{fafa}
    S_{k} (\sigma_p) \in (-4 \bar C , 2 \bar C ), \text{ for }  k = 1,2,3,4.
 \end{equation}

Next suppose that $\mathbf{u }=O $ and $\mathbf{v}=(n, 0) \in \Z^2$ with  $n>0 $. We will construct explicitly a path $\gamma^{\flat} \in \Gamma_{O, \mathbf{v}}$ that satisfies
\begin{equation}
\label{bound}
\sup_{n\leq |\gamma^{\flat}|} |S_n (\gamma^{\flat} )|\leq 6\bar C.
\end{equation}
The construction of the path is done by recursion over $k$, $\gamma^{\flat}_k$ being the initial part of the path, joining $O$ with $(k,0)$.
\begin{itemize}
\item[1.] Let $\gamma^{\flat}_1$ be the edge joining $O$ to  $(1,0)$.
\item[2.] Given the path $\gamma^{\flat}_k$, for any $k=1,2,\ldots,n-1$, we form the concatenation $\gamma^{\flat}_{k+1}=\gamma^{\flat}_k \odot \tau_k$ in the following way:
\begin{itemize}
\item[2.a.] If
$$
Z_{ ((k , 0) ,     (k+1, 0) ) }  X_{ \{(k , 0) ,     (k+1, 0) \} }   S(\gamma^{\flat}_k)\leq 0
$$
then $ \tau_k=  \{(k , 0) ,     (k+1, 0)\}     $.
\item[2.b.]  If
$$
Z_{( (k , 0) ,     (k+1, 0) ) }  X_{ \{(k , 0) ,     (k+1, 0) \} }    S (\gamma^{\flat}_k)    > 0
$$
then define the sign variable $\xi_k=-\text{sign}(S (\gamma^{\flat}_k)S(\sigma_k))$ and set
$$
\tau_k= (\xi_k\sigma_k)^{\odot i_k}\odot { ((k , 0) ,     (k+1, 0) ) },
$$
where $i_k$ is the smallest integer $i$ such that
$$
S (\gamma^{\flat}_k)\left (S (\gamma^{\flat}_k)+i\xi_kS(\sigma_k)\right) \leq 0.
$$
\end{itemize}
\end{itemize}

\medskip \medskip

In order to prove the bound \eqref{bound} we start by proving that
$$
|S(\gamma^{\flat}_k)|\leq 4 \bar C ,
$$
for any integer $ k $. This is certainly true for $k =1$. Now suppose that this is true for a certain $k $ and let us prove it for $k+1$.
If $2.a.$ holds for $k $ this is trivial. If $2.b.$ holds for $k $ suppose w.l.o.g. that $S(\gamma^{\flat}_k) >0 $.
Then for any $i= 1 , \ldots , i_k -1 $ one has that
$$
          S( \gamma^{\flat}_k \odot     (\xi_k\sigma_k)^{\odot i}      ) = S(\gamma^{\flat}_k)+i\xi_kS(\sigma_k) \in ( 0, S( \gamma^{\flat}_k )   ) ,
$$
therefore
$$
 S( \gamma^{\flat}_{k +1} )= S (\gamma^{\flat}_k)+i_k\xi_kS(\sigma_k) + Z_{ ((k , 0) ,     (k+1, 0) ) }  X_{ \{(k , 0) ,     (k+1, 0) \} } \in (-4 \bar C , \bar C ) .
$$
Taking into account \eqref{fafa} for the intermediate steps of the cycle $ \sigma_k  $ one has the
desired inequality \eqref{bound}.

The argument can be continued on a vertical path; as a consequence the bound \eqref{bound} always holds
for a suitable path joining any two vertices $\mathbf{u},\mathbf{v} \in \Z^2$. \end{proof}

\bibliographystyle{abbrv}


\end{document}